\documentclass{amsart}
\usepackage{graphicx}

\newtheorem{theorem}{Theorem}[section]
\newtheorem{proposition}[theorem]{Proposition}
\newtheorem{lemma}[theorem]{Lemma}
\theoremstyle{definition}
\newtheorem{definition}[theorem]{Definition}
\newtheorem{example}[theorem]{Example}
\newtheorem*{acknowledgements}{Acknowledgements}

\theoremstyle{remark}
\newtheorem{remark}[theorem]{Remark}

\newcommand{\quadmatrix}[4]{\left(\begin{array}
{cc}#1&#2\\#3&#4\end{array}\right)}

\numberwithin{equation}{section}

\begin{document}

\title[Discrete flat and linear Weingarten 
surfaces in $\mathbb{H}^3$]{Discrete flat surfaces and linear Weingarten 
surfaces in hyperbolic 3-space}

\author{Tim Hoffmann}
\address{Department of Mathematics, Munich Technical University, 
85748 Garching, Germany}
\email{tim.hoffmann@ma.tum.de}

\author{Wayne Rossman}
\address{Department of Mathematics, Kobe University,
Kobe 657-8501, Japan}
\email{wayne@math.kobe-u.ac.jp}

\author{Takeshi Sasaki}
\address{Department of Mathematics, Kobe University,
Kobe 657-8501, Japan}
\email{sasaki@math.kobe-u.ac.jp}

\author{Masaaki Yoshida}
\address{Department of Mathematics, Kyushu University,
Fukuoka 819-0395, Japan}
\email{myoshida@math.kyushu-u.ac.jp}


\date{\today}


\begin{abstract}
We define discrete flat surfaces in hyperbolic $3$-space 
$\mathbb{H}^3$ from the perspective of 
discrete integrable systems and prove properties that justify 
the definition.  We show how these surfaces correspond 
to previously defined discrete constant mean curvature $1$ 
surfaces in $\mathbb{H}^3$, and we also describe 
discrete focal surfaces (discrete caustics) that can be used to 
define singularities on discrete flat surfaces.  Along the way, 
we also examine discrete linear Weingarten surfaces of Bryant type 
in $\mathbb{H}^3$, and consider an example of a discrete flat surface related 
to the Airy equation that exhibits swallowtail singularities and 
a Stokes phenomenon.  
\end{abstract}

\maketitle


\section{Introduction}
\label{intro}

The classical Weierstrass representation for minimal surfaces in 
Euclidean $3$-space $\mathbb{R}^3$ gives a local conformal 
parametrization for any minimal surface.  It involves choosing 
two holomorphic functions (or perhaps meromorphic functions when 
considering the surfaces more globally) on a Riemann surface.  
If one restricts to isothermic parametrizations, that is, conformal 
parametrizations that are also curvature line coordinate 
systems, then 
the representation reduces to the choice of just one holomorphic 
function.  Since every minimal surface has local isothermic coordinates 
(away from umbilics), this reduction does not involve any loss of 
generality beyond avoiding umbilic points.  

Once one restricts to isothermic parametrizations, it becomes 
possible to give a definition for discrete analogs of minimal surfaces 
\cite{BP}.  These analogs are ``discrete isothermic'' meshes (a 
definition of this is given later in this paper).  They are 
comprised of planar quadrilaterals, which in particular have 
concircular vertices.  

By a transformation called the Lawson correspondence or $T$-transformation 
or Calapso transformation \cite{Udo2}, one can 
produce all constant mean curvature (CMC) 
$1$ surfaces in hyperbolic $3$-space $\mathbb{H}^3$ from minimal surfaces in 
$\mathbb{R}^3$.  There is a corresponding holomorphic 
representation for those surfaces as well, 
first given by Bryant \cite{Br}.  
Correspondingly, without loss of generality beyond avoiding umbilics, 
one can restrict to isothermic coordinates in this case also, 
and one has a discrete analog of CMC $1$ surfaces in $\mathbb{H}^3$, first 
found by Hertrich-Jeromin \cite{Udo1}.  

In the case of smooth surfaces there is also a holomorphic 
representation for flat (i.e. intrinsic curvature zero) 
surfaces in $\mathbb{H}^3$ \cite{GMM-first} and this also ties in to the 
above-mentioned Bryant representation, 
as there are deformations from CMC $1$ surfaces in $\mathbb{H}^3$ 
to flat surfaces via a family of linear Weingarten surfaces in $\mathbb{H}^3$ 
\cite{GMM}.  These do not include all linear Weingarten surfaces, 
but rather a certain special subclass called linear Weingarten 
surfaces of {\em Bryant type} \cite {GMM} \cite{KU}, 
so named because they have Bryant-type representations.  

Thus it is natural to wonder if flat surfaces also have a discrete 
analog, and we will see here that they do.  Once this discrete 
analog is found, a new question about "singularities on discrete flat 
surfaces" naturally presents itself, in this way: Unlike the 
smooth isothermic 
minimal surfaces in $\mathbb{R}^3$ and CMC $1$ surfaces in 
$\mathbb{H}^3$, smooth flat fronts 
have certain types of 
singularities, such as cuspidal edges and swallowtails (in 
fact, indirectly, this is what the naming "fronts" -- rather than 
"surfaces" -- indicate).  The means for 
recognizing where the singularities are on smooth flat fronts are 
clear, and one can make classifications of those 
surfaces' most generic types of 
singularities just from looking at the choices of holomorphic 
functions used in their representation \cite{KRSUY}.  However, 
in the case of discrete flat surfaces, it is not a priori clear where 
the singularities are, nor even what such a notion would 
mean.  Since one does not have first and second fundamental 
forms at one's disposal in the discrete case, one must find an alternate 
way of defining singularities.  We aim towards this by defining and 
using a discrete analog of caustics, also called focal surfaces, for 
smooth flat fronts.  For a smooth flat front, the caustic is the surface 
comprised of all the singular points on all parallel surfaces of that 
flat front.  (The parallel surfaces are also flat.)  
Thus the singular set of the flat front can be retrieved 
by taking its intersection with its caustic.  In the case of 
a smooth flat front, the caustic is again a flat surface, but this will 
not quite be the case for discrete flat surfaces.  

We will also present a number of examples of these discrete 
flat surfaces.  
In addition to the rather simple examples of discrete cylinders and 
discrete surfaces of revolution, we will also discuss a discrete flat 
surface based on the Airy equation.  This example exhibits 
swallowtail singularities and a 
Stokes phenomenon, similar to that of the analogous surface in the 
smooth case, as shown by two of the authors in \cite{SY}.  This 
last example hints at existence of a robust collection of discrete flat 
surfaces with interesting geometric properties yet to be explored.  

Thus, the purpose of this paper is to: 
\begin{enumerate}
\item provide a definition for discrete flat surfaces and 
      discrete linear Weingarten surfaces of Bryant type 
      in hyperbolic $3$-space $\mathbb{H}^3$; 
\item give properties of these surfaces that justify our choice of 
      definitions (in particular, as smooth flat fronts have extrinsic 
      curvature $1$, we identify notions of discrete extrinsic curvature 
      of discrete flat surfaces which do indeed attain the value $1$);
\item show that these surfaces have concircular quadrilaterals;
\item study examples of these surfaces, and in particular look at 
      swallowtail singularities and 
      global properties of an example related to the Airy equation;
\item give a definition of discrete caustics for discrete flat surfaces;
\item show that the caustics also have concircular quadrilaterals 
      and that they provide a means for identifying a notion of 
      singularities on discrete flat surfaces.  
\end{enumerate}
In Section \ref{section1} we describe smooth and discrete minimal 
surfaces in Euclidean $3$-space $\mathbb{R}^3$, to help motivate later definitions, 
and we also give the definition of a discrete holomorphic function, which 
will be essential to everything that follows.  
In Section \ref{section2} we describe smooth CMC $1$ surfaces, 
and flat surfaces and linear Weingarten surfaces of Bryant type 
in $\mathbb{H}^3$, again as motivational material for the 
definitions of the corresponding discrete surfaces in Section 
\ref{section3}.  We prove in Section \ref{section3} 
that discrete flat surfaces and linear Weingarten surfaces of Bryant type 
have concircular quadrilaterals.  Also, Section \ref{section3} 
provides a natural representation for discrete flat surfaces which 
gives the mapping of the surfaces as products of $2$ by $2$ matrices 
times their conjugate transposes, and we show that this representation 
applies to the case of discrete CMC $1$ surfaces as well.  The definition 
for discrete CMC $1$ surfaces is already known \cite{Udo1}, but 
the representation here for those 
surfaces is new.  In Section \ref{section4} we look at a specific 
discrete example whose smooth analog is equivalent to solutions of 
the Airy equation, and we look at the asymptotic behavior of that surface, 
which exhibits swallowtail singularities and 
a Stokes phenomenon.  In Section 
\ref{section5}, we look at normal lines to discrete flat surfaces.  
With this we can do several things.  For example, we look at parallel 
surfaces (which are also discrete flat) and show that the area of 
corresponding quadrilaterals of the normal map equals the area of 
the quadrilaterals of the surface itself, as should be expected, 
since in some sense the extrinsic curvature of the surface 
is identically equal to $1$ (note that the analogous statement is 
true for smooth surfaces with extrinsic curvature $1$, infinitesimally).   
Then, using distances from the surface's vertices to the intersection 
points of the normal lines, we consider a 
discrete analog of the extrinsic curvature and see that it is $1$ in the 
discrete case as well.  Furthermore, those intersections give us a 
means to define discrete caustics, and, as mentioned above, 
we use those caustics to study the nature of "singularities" on 
discrete flat surfaces, in the final Section \ref{section6}.  

\begin{acknowledgements}
The authors thank Udo Hertrich-Jeromin for 
fruitful discussions and valuable comments.
\end{acknowledgements}

\section{Smooth and discrete minimal surfaces in $\mathbb{R}^3$}
\label{section1}

A useful choice of coordinates for a surface is 
isothermic coordinates.  Not all surfaces have such coordinates, but CMC 
surfaces in space forms such as $\mathbb{R}^3$ and 
$\mathbb{H}^3$ do have them, 
away from umbilic points.  Isothermic coordinates will be 
of central importance in this paper.  

Another useful tool in the study of surfaces in space forms is 
the Hopf differential, which is defined as $Q = 
\langle f_{zz},N \rangle dz^2$, where the surface $f$ is a 
map from points $z$ in a portion of the complex plane $\mathbb{C}$, 
$\langle \cdot , \cdot \rangle$ is the bilinear extension of the 
metric for the ambient space form to complex vectors, and $N$ is the 
unit normal to the surface.  
When the coordinate $z$ is conformal and the surface is CMC, then $Q$ 
will be holomorphic in $z$.  Umbilic points of the surface occur 
precisely at the zeros of the Hopf differential.  

\subsection{The 
Weierstrass representation for smooth minimal surfaces}
Locally, away from umbilics, 
we can always take a smooth minimal immersion $f=f(x,y)$ 
into $\mathbb{R}^3$ to have isothermic coordinates $(x,y)$ in a domain of 
$\mathbb{R}^2$.  
Let $N$ denote the unit normal vector to $f$.  Then, setting $z=x+iy$, 
the Hopf differential becomes $Q = r dz^2$ for some real constant 
$r$, and rescaling the coordinate $z$, we may assume $r=1$.  

Let $g$ be the stereographic 
projection of the Gauss map $N$ to the complex plane, and set $g^\prime = 
dg/dz$.  As we are only concerned with the local behavior of the surface, 
and we are allowed to replace the surface with any rigid motion of it, 
we may ignore the possibility that $g$ has poles or other singularities, 
and so the map $g:\mathbb{C} \to \mathbb{C}$ is holomorphic.  Because we avoid umbilic points 
of $f$, we also know that $g^\prime$ is never zero.  Thus the Weierstrass 
representation is (with $\sqrt{-1}$ regarded as lying in the complex plane 
$\mathbb{C}$) 
\[ f = \mbox{Re} \int_{z_0}^z ( 2 g, 1-g^2,\sqrt{-1} (1+g^2) ) \omega 
\; , \;\;\; 
\omega = \frac{Q}{dg} = \frac{dz}{g^\prime} \; . \]  
Associating $(1,0,0)$, $(0,1,0)$ and $(0,0,1)$ with 
the quaternions $i$, $j$ and $k$, respectively, we have 
\begin{equation}\label{the-Udo-way} 
f_x = (i-gj) j \frac{1}{g_x} (i-g j) \; , \;\;\; 
f_y = (i-gj) j \frac{-1}{g_y} (i-g j) \; . \end{equation} 
We have converted to a formulation using quaternions here, because 
this type of formulation has been used to define discrete minimal 
surfaces in $\mathbb{R}^3$ and discrete CMC $1$ surfaces in 
$\mathbb{H}^3$, and we 
wish to make comparisons to those formulations.  

Note that by restricting to isothermic coordinates, we can then determine 
minimal surfaces by choosing just one holomorphic function $g$.  

\subsection{Discrete holomorphic functions}\label{sect:2pt2}
To define discrete minimal surfaces, we use discrete holomorphic 
functions $g=g_{m,n}:D \to \mathbb{C}$, where $D$ is the square 
integer lattice $\mathbb{Z}^2$, or a subdomain of it.  
Discrete holomorphic functions are defined as 
follows: defining the cross ratio of $g$ to be 
\[ \text{cr}_{m,n} = (g_{m+1,n}-g_{m,n}) (g_{m+1,n+1}-g_{m+1,n})^{-1} 
(g_{m,n+1}-g_{m+1,n+1}) (g_{m,n}-g_{m,n+1})^{-1} \; , \]
we say that $g$ is {\em discrete holomorphic} 
if there exists a discrete mapping $\alpha$ to $\mathbb{R}$ such that 
\begin{equation}\label{rats2} \text{cr}_{m,n} = 
\frac{\alpha_{(m,n)(m+1,n)}}{\alpha_{(m,n)(m,n+1)}} < 0 \; , 
\end{equation} 
with $\alpha_{(m,n)(m+1,n)} = \alpha_{(m,n+1)(m+1,n+1)}$ and 
$\alpha_{(m,n)(m,n+1)} = \alpha_{(m+1,n)(m+1,n+1)}$ for all 
quadrilaterals (squares with edge length $1$ and vertices in $D$.)  
See \cite{bp2}.  We call 
the discrete map $\alpha$ a {\em cross ratio factorizing 
function} for $g$.  

Note that $\alpha$ is defined on edges of $D$, not vertices.  
Note also that $\alpha$ is symmetric, that is, 
$\alpha_{(m,n,)(m+1,n)} = \alpha_{(m+1,n)(m,n)}$ 
and 
$\alpha_{(m,n,)(m,n+1)} = \alpha_{(m,n+1)(m,n)}$.  

There is a freedom of a single real factor in the choice of these 
$\alpha_{(m,n)(m+1,n)}$ and $\alpha_{(m,n)(m,n+1)}$, 
since we could replace 
all of them with $\lambda \alpha_{(m,n)(m+1,n)}$ and 
$\lambda \alpha_{(m,n)(m,n+1)}$ 
for any nonzero real constant $\lambda$, and all relevant properties 
would still hold.  Throughout this paper we 
use $\lambda$ to denote that free factor.  

In the above definition of the cross ratio, we have a product of 
four terms.  Since $g_{m,n} \in \mathbb{C}$, these terms all commute, and 
so we could have written this cross ratio simply as a product of 
two fractions.  However, when we later consider the cross 
ratio for quaternionic-valued objects or matrix-valued objects, 
commutativity no longer holds and the order of the product in 
the cross ratio becomes vital.  So, for later reference, we have 
chosen to write the cross ratio in the somewhat cumbersome way above.  

The definition above for discrete holomorphic functions is in the "broad" 
sense.  The definition in the "narrow" sense would be that 
$\text{cr}_{m,n}$ is identically $-1$ on $D$ (see 
\cite{BP}, \cite{bp2}).  
Furthermore, note that, unlike the case of smooth holomorphic functions, 
the discrete derivative or discrete integral of a discrete 
holomorphic function is generally not another 
discrete holomorphic function.

Let us exhibit some examples of discrete holomorphic functions:
\begin{enumerate}
\item Let $D=\mathbb{Z}^2 = \{ (m,n) \, | \, m,n \in \mathbb{Z} \}$, 
      and set $g_{m,n} = c (m+i n)$ for $c$ a complex constant.
\item Let $D=\mathbb{Z}^2$, 
      and set $g_{m,n} = e^{c (m+i n)}$ for $c$ a real or pure 
      imaginary constant. 
      One could also take the function $e^{c_1 m+i c_2 n}$ for 
      choices of real constants 
       $c_1$ and $c_2$ so that the cross ratio is identically $-1$, 
      giving a discrete holomorphic function in the narrow sense.
\item\label{item5} In Section \ref{section4} we will describe a 
      discrete flat surface based on a 
      discrete version of the power function $g=z^\gamma$ 
      ($\gamma \in \mathbb{R}$), which we define here.  This function 
      is discrete holomorphic in the narrow sense.  
      It is defined by the recursion 
\begin{equation}\label{eqn:alphagnm} 
\gamma \cdot g_{m,n} = 2 m 
\frac{(g_{m+1,n}-g_{m,n})(g_{m,n}-g_{m-1,n})}{g_{m+1,n}-g_{m-1,n}}
+ 2 n 
\frac{(g_{m,n+1}-g_{m,n})(g_{m,n}-g_{m,n-1})}{g_{m,n+1}-g_{m,n-1}} 
\; . \end{equation}
We start with $D = \{ (m,n) \, | \, m,n \geq 0 \}$.  
For $\gamma \in (0,2)$, the initial conditions should be 
\[ g_{0,0}=0 \; , \;\;\; g_{1,0} = 1 \; , \;\;\; g_{0,1} = i^\gamma \; . \]  
We can then use \eqref{eqn:alphagnm} to 
propagate along the positive axes $\{ g_{m,0} \}$ 
and $\{ g_{0,n} \}$ with $m > 1$ and $n>1$, respectively.  
We can then compute general $g_{m,n}$ (for both $m>0$ and $n>0$) by 
using that the 
cross ratio is always $-1$.  The $g_{m,n}$ will then 
automatically satisfy the recursion relation \eqref{eqn:alphagnm}.  
This definition of the discrete power function can be found in 
Bobenko \cite{Bob-new}.  (It is also found in a recently published 
textbook \cite{BS}.)  Agafonov \cite{Ag} showed that these discrete 
power functions are embedded in wedges 
(see Figure \ref{fct43}), and 
are Schramm circle packings (see \cite{Sch}).  Note that, for 
$m \in \mathbb{Z}$ and $m \geq 1$, 
\begin{equation}\label{eqn-agafonovaxes} g_{2m,0} = 
\dfrac{-m \left(\tfrac{\gamma}{2} 
\right)_{m}}{\left (-\tfrac{\gamma}{2} \right)_{m+1}} \; , \;\;\; 
   g_{2m+1,0} = \dfrac{-\left( \tfrac{\gamma}{2} 
\right)_{m+1}}{\left( -\tfrac{\gamma}{2} \right)_{m+1}} 
\; , \;\;\;    g_{0,n} = i^\gamma g_{n,0} \; , 
\end{equation} 
where $(a)_m = a (a+1)...(a+m-1)$ 
denotes the Pochhammer symbol, 
and a closed expression for general $g_{m,n}$ is still unknown.  
We explore this difference 
equation \eqref{eqn:alphagnm} in more detail in Appendix \ref{appendix} 
at the end of this paper.  
\end{enumerate}

\begin{figure}[h]
\begin{center}
\begin{tabular}{cc}
 \includegraphics[width=4cm]{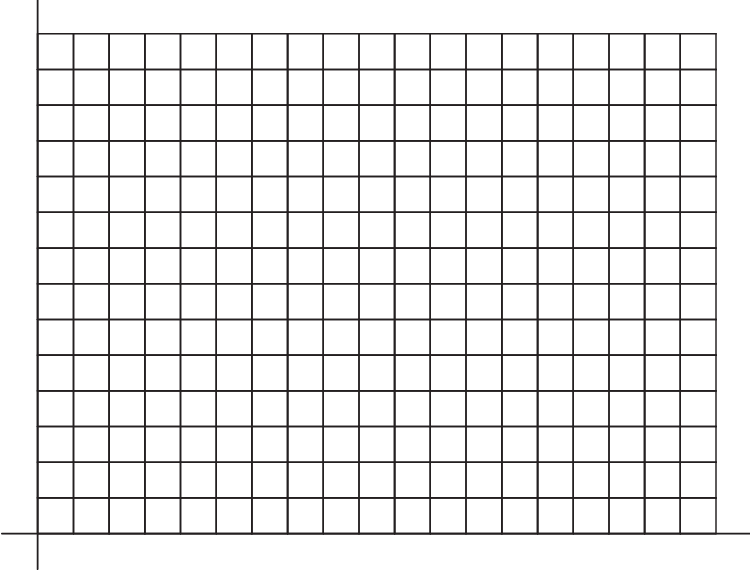} \qquad &
\qquad  \includegraphics[width=4cm]{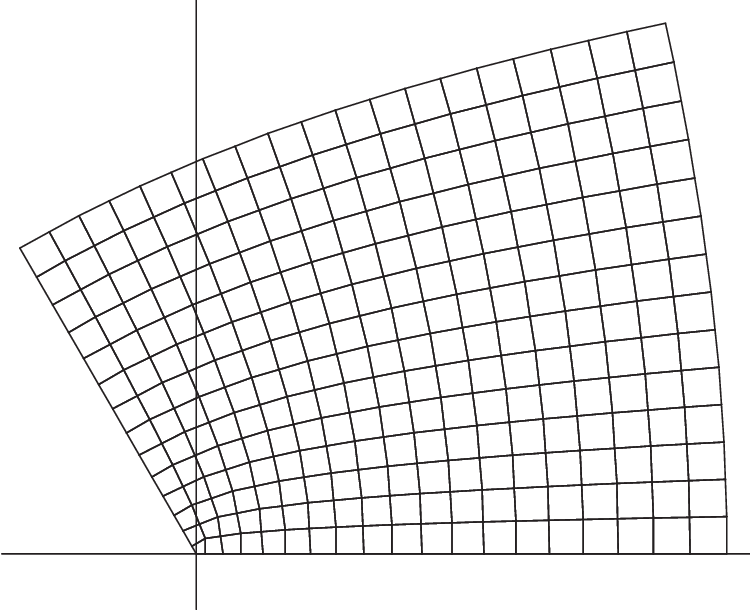} 
\end{tabular}
\end{center}
\caption{Domain (left) and image (right) 
for the discrete power function $z^{4/3}$.}
\label{fct43}
\end{figure}

\subsection{Discrete minimal surfaces}\label{sect:2pt3}
The representation \eqref{the-Udo-way} for smooth minimal surfaces above 
suggests that the definition for discrete minimal surfaces is (see 
\cite{Udo1}) 
\begin{equation}\label{starstar} 
f_q - f_p = (i-g_p j) j \frac{\alpha_{pq}}{g_q-g_p} (i-g_q j) \; , 
\end{equation} 
where $g:D \to \mathbb{C}$ is a discrete holomorphic function with cross ratio 
factorizing function $\alpha$, and 
$p$ and $q$ are either $(m,n)$ and $(m+1,n)$, or $(m,n)$ and $(m,n+1)$.  
This defines the surface $f$ up to translations of 
$\mathbb{R}^3$.  The freedom 
of scaling of $\alpha$ leads to homotheties of $f$.  

As in the smooth case, where we avoided umbilics, and thus $g^\prime$ 
was never zero, we will make the following assumption throughout this paper: 
\[ \text{Assumption:} \;\;\; g_q-g_p \neq 0 \; . \]

\begin{example}
The discrete holomorphic function $c (m+in)$ for $c$ a complex constant 
will produce a minimal surface called a discrete Enneper surface, 
and graphics for this surface can be seen in \cite{BP}.  
\end{example}

\begin{example}
The discrete holomorphic function $e^{c_1 m+i c_2 n}$ for 
choices of constants $c_1$ and $c_2$ so that the cross ratio is 
identically $-1$ will produce a minimal surface called a discrete 
catenoid, and graphics for this surface also can be seen in \cite{BP}.  
\end{example}

\section{Smooth CMC $1$ surfaces, flat fronts, and 
linear Weingarten surfaces in $\mathbb{H}^3$}\label{section2}

\subsection{Smooth CMC $1$ surfaces}\label{section3pt1}
Similarly to the case of minimal surfaces, we can 
describe smooth and discrete CMC $1$ surfaces in 
$\mathbb{H}^3$.  Hyperbolic $3$-space 
$\mathbb{H}^3$, considered in Minkowski 
$4$-space 
$\mathbb{R}^{3,1} = \{(x_0,x_1,x_2,x_3) \, | \, x_j \in 
\mathbb{R} \}$ 
(with Minkowski metric $-dx_0^2+dx_1^2+dx_2^2+dx_3^2$), is 
\[ \mathbb{H}^3 = \{ (x_0,x_1,x_2,x_3) \in \mathbb{R}^{3,1} \, | 
\, x_0>0 , x_0^2-x_1^2-x_2^2-x_3^2 = 1 \}
\approx \]\[ \left\{ \left. X = 
\begin{pmatrix} x_0+x_3 & x_1+ix_2 \\ x_1-ix_2 & 
x_0-x_3 \end{pmatrix} \right| \, \text{tr}(X)>0, \det X = 1 \right\} 
= \{ F \cdot \bar F^T \, | \, F \in SL_2(\mathbb{C}) \} \; , \] where the 
superscript $T$ denotes transposition.  
 
A smooth isothermically-parametrized CMC $1$ surface (away from 
umbilic points), has the Bryant equation \cite{Br} 
\begin{equation}\label{eqn:bryant} 
dF = F \begin{pmatrix}  g & -g^2 \\ 1 & -g \end{pmatrix} 
\frac{dz}{g^\prime} \; , \;\;\; F \in SL_2(\mathbb{C}) \; , \end{equation} 
where $g$ is a holomorphic function with nonzero derivative, 
and the surface is then \[ f_1 = F \cdot \bar F^T \in \mathbb{H}^3 \; . \]  

\subsection{Smooth flat fronts}\label{section3pt2}
Starting with a smooth CMC $1$ surface $f_1$ with lift $F$ as above, define 
\begin{equation}\label{EinrelationtoF} 
E = F \cdot \begin{pmatrix} 1 & g \\ 0 & 1 \end{pmatrix} \; . \end{equation}
A flat front is then given by 
\[ f_0 = E \cdot \bar E^T \in \mathbb{H}^3 \; , \]
with unit normal vector field 
\begin{equation}\label{flatsurfnormal}
N = E \cdot \begin{pmatrix} 1 & 0 \\ 0 & -1 \end{pmatrix}
\cdot \bar E^T \; . \end{equation}
We know this surface $f_0$ is flat, because (see \cite{GMM}) 
\begin{equation}\label{E-eqn-for-flat-guys} 
dE = E \begin{pmatrix}  0 & g^\prime \\ (g^\prime)^{-1} & 0 
\end{pmatrix} dz \; . \end{equation} 
This surface with singularities is actually a front, 
because $|g^\prime|^{2} + |g^\prime|^{-2} > 0$, which means 
that the associated 
Sasakian metric is positive definite.  (See Theorem 2.9 in 
\cite{KUY}.)  The notion of fronts is important in the study of smooth 
flat surfaces with singularities in 
$\mathbb{H}^3$.  For example, it is a 
necessary notion for considering the caustic of a flat surface 
with singularities.  However, it will not play such a direct 
role in our considerations on discrete surfaces here.  Thus, from 
here on out, we will simply consider smooth flat surfaces with 
singularities, and sometimes will even just call them {\em flat 
surfaces} even though they might have singularities.  When the 
front property is actually playing a role, we shall parenthetically 
refer to the word "front".  For more 
information about flat fronts, see \cite{KRSUY}, \cite{KRUY1} 
and \cite{KUY}.  

\begin{remark}\label{KRSUY-remark-on-sings}
Because the off-diagonal terms $g^\prime$ and $(g^\prime)^{-1}$ 
in \eqref{E-eqn-for-flat-guys} 
are inverse to each other, the conditions in \cite{KRSUY} for having 
singular points, cuspidal edges and swallowtails 
on $f_0$ simplify to this: 
\begin{enumerate}
\item singular points occur precisely at points where $|g^\prime| = 1$; 
\item a singular point is a cuspidal edge if and only if 
$\text{Im}(\tfrac{g^{\prime\prime}}{g^\prime}) 
\neq 0$ ("Re" and "Im" denote the real and imaginary parts, 
respectively); 
\item a singular point is a swallowtail if and only if 
$g^{\prime\prime} \neq 0$ and 
$\text{Im}(\tfrac{g^{\prime\prime}}{g^\prime}) = 0$ and 
$\text{Re}((\tfrac{g^{\prime\prime}}{g^\prime})^\prime) \neq 0$.   
\end{enumerate}
The condition that 
$\text{Im}(\tfrac{g^{\prime\prime}}{g^\prime}) = 0$ holds 
at some point along a singular curve $|g^\prime|=1$ is equivalent 
to the curve having a vertical tangent line at that point.  
\end{remark}

\subsection{The hyperbolic Schwarz map and 
a special coordinate $w$}\label{remark-on-q}
For later reference, we can take a new coordinate $w$ such that \[
dw = \frac{1}{g^\prime} dz \; , \] which is locally well-defined, 
and still conformal, although not necessarily isothermic.  This gives 
\begin{equation}\label{eqn:q} 
E^{-1} dE = \begin{pmatrix} 0 & q \\ 1 & 0 \end{pmatrix} dw \; , 
\end{equation} with 
\[ q|_{w(z)} = (g^\prime)^2 = \left. \frac{dg}{dw}\right|_{w(z)} \; . 
\]  

The reason for changing variables from $z$ to $w$ is 
that Equation \eqref{eqn:q} now becomes 
\begin{equation}\label{eqn:q2}
\tfrac{d^2}{dw^2} u - q(w) \cdot u = 0 
\end{equation}
with $f_0 = E \bar E^T$ the hyperbolic Schwarz map (see \cite{SY}), where 
\begin{equation}\label{eqn:star-s} E = \begin{pmatrix}
u_1 & \tfrac{d}{dw} u_1 \\ u_2 & \tfrac{d}{dw} u_2
\end{pmatrix} \; , \end{equation} 
with functions $u_1$, $u_2$ that are linearly independent solutions of 
Equation \eqref{eqn:q2} chosen so that the constant 
$u_1 \tfrac{d}{dw}(u_2)-u_2 \tfrac{d}{dw}(u_1)$ will 
be $1$.  This equation 
\eqref{eqn:q2} with $q(w)=w$ is the well-known Airy equation.  

\begin{remark}
Using $q$ and $w$, the conditions in Remark \ref{KRSUY-remark-on-sings} 
become: 
\begin{enumerate}
\item Singular points: $|q|=1$, 
\item Cuspidal edge points: 
$\text{Im}(\tfrac{q_w}{q^{3/2}}) \neq 0$, 
\item Swallowtail points: 
$q_w \neq 0$, 
$\text{Im}(\tfrac{q_w}{q^{3/2}}) = 0$
and 
$\text{Re}(\tfrac{2 q_{ww} q - 3 q_w^2}{q^{3}}) \neq 0$.  
\end{enumerate}
\end{remark}

\subsection{Smooth linear Weingarten surfaces of Bryant 
type}\label{smoothLinWein}
We will now give a deformation through linear Weingarten surfaces 
between the surfaces $f_1$ and $f_0$ described in Sections 
\ref{section3pt1} and \ref{section3pt2}.  This 
deformation was first introduced in \cite{GMM}.  There are numerous 
ways to choose the deformation, and no one way is geometrically 
more canonical than any other.  We will soon come back to this 
issue (Section \ref{rem:ambiguityofGMM}).  However, for now we 
will simply fix one choice for the 
deformation -- the one that deforms the $f_1$ and $f_0$ as given 
in Sections \ref{section3pt1} and \ref{section3pt2} into each other, in 
accordance with the notations of previous papers (\cite{KRSUY}, 
\cite{KRUY1}, \cite{KRUY2}, \cite{KUY}).  

This particular choice 
will suffice when we switch to investigating discrete surfaces later.  
In fact, the resulting classes of discrete 
flat surfaces and discrete linear Weingarten surfaces 
of Bryant type, although defined in terms of the deformation, do not 
actually depend on the choice of 
deformation.  (We say more about this in 
Section \ref{section:dlwsobt}.)  
Hence the theorems we prove about these surfaces 
also are independent of the choice of deformation.  

We shall refer to this choice of deformation as the 
"first Weingarten family" (of either $f_1$ or $f_0$), 
and the procedure that we follow 
for constructing it is as follows.  Following \cite{GMM}, 
\begin{itemize}
\item we convert the $E$ in \cite{GMM}, actually notated as 
``$g$'' there, to $(E^T)^{-1}$, 
\item then changing the holomorphic function 
$h$ in \cite{GMM} to $-g$ for the function $g$ given here, 
\item and allowing 
$-\omega$ in \cite{GMM} to become $dz/g^\prime$ for the function 
$g$ here, 
\item and also changing $f$ and $N$ to $(f^T)^{-1}$ and 
$(N^T)^{-1}$, respectively, 
\end{itemize}
a linear Weingarten 
surface of Bryant type in $\mathbb{H}^3$ satisfying 
(see also \cite{KU}) 
\begin{equation}\label{eqn:linweineqn} 
2 t (H-1)+(1-t) K = 0 \; , \end{equation} 
where $H$ and $K$ are the mean and 
intrinsic curvatures, respectively, is 
\[ f_t = (EL)\overline{(EL)}^T \; . \]  Here $E$ satisfies 
\eqref{E-eqn-for-flat-guys} and 
\[ L = \begin{pmatrix} \beta & -t g \beta \\ 0 & \beta^{-1} \end{pmatrix} 
\; , \;\;\; \beta = 
\sqrt{\frac{1 + t g \bar g}{1 + t^2 g \bar g}} \in \mathbb{R} \; . \] 
All linear Weingarten surfaces satisfying \eqref{eqn:linweineqn} 
(i.e. of Bryant type) can be constructed in this way (see \cite{GMM}).  

When $t = 0$, we use the frame $E = E \cdot (L|_{t=0}) = 
F \cdot \begin{pmatrix} 1 & g \\ 0 & 1 \end{pmatrix}$, 
giving a flat surface $f_0$, as in Section \ref{section3pt2}.  
When $t = 1$, we use the frame $F = E \cdot (L|_{t=1}) = 
E \cdot \begin{pmatrix} 1 & -g \\ 0 & 1 \end{pmatrix}$, 
giving a CMC $1$ surface $f_1$, as 
in Section \ref{section3pt1}.  Thus we have a deformation 
through linear Weingarten surfaces in 
$\mathbb{H}^3$, from CMC $1$ surfaces in 
$\mathbb{H}^3$ to flat surfaces in $\mathbb{H}^3$.  
(See also \cite{KU}.)  

\subsection{Geometric non-uniqueness 
of the deformation in \cite{GMM}}
\label{rem:ambiguityofGMM}
We now explain in more detail why there is non-uniqueness 
for the choice of deformation through linear Weingarten surfaces 
of Bryant type.  The deformation between smooth flat surfaces and 
smooth CMC $1$ surfaces through linear Weingarten surfaces, 
given in Section \ref{smoothLinWein}, is not 
uniquely determined in any geometric sense, because of ambiguities 
in the choice of Weierstrass data.  We illustrate this with 
two lemmas, both of which are easily verified:  

\begin{lemma}\label{lem:3pt6}
Given a smooth isothermically-parametrized CMC $1$ surface $f_1$ in 
$\mathbb{H}^3$ with lift $F$ and Weierstrass data $g$, the transformation 
\[ F \to F \cdot B \; , \;\;\; B = 
\begin{pmatrix}
p & q \\ -\bar q & \bar p
\end{pmatrix} \in \text{SU}_2 
\] will not change the resulting surface $f_1 = F \cdot \bar F^T = 
FB \cdot \overline{FB}^T$, and will change the Weierstrass data 
by \[ g \to \hat g = \frac{\bar p g - q}{\bar q g + p} \; . \]
\end{lemma}

\begin{remark}
By Lemma \ref{lem:3pt6}, different choices for $B$ do not affect $f_1$.  
However, when $B$ is not diagonal, the transformation in the above lemma
generally will result in a different deformation $f_t$ through 
linear Weingarten surfaces for $t < 1$, and also in a different 
flat surface $f_0$.
\end{remark}
  
\begin{lemma}\label{lem:3pt7}
Given a flat surface $f_0$ with Weierstrass data $g^\prime$ and lift 
$E$ as in Equation \eqref{E-eqn-for-flat-guys}, then the transformation 
\[ g \to g+a \; , \] 
followed by the transformation 
\[ E \to E \cdot B \; , \] where $B \in \text{SU}_2$ is 
either diagonal or off-diagonal and 
$a$ is any complex constant, will not change the resulting surface 
$f_0 = E \cdot \bar E^T = EB \cdot \overline{EB}^T$.  
\end{lemma}

\begin{remark}
Under the transformations of $E$ and $g$ in Lemma \ref{lem:3pt7}, 
$f_0$ is not affected.  However, when $B$ is off-diagonal or $a$ is 
not zero, these transformations 
generally will result in a different deformation $f_t$ through 
linear Weingarten surfaces for $t > 0$.  In particular, the CMC $1$ 
surface $f_1$ generally will change.  
\end{remark}

The two lemmas and two remarks above show that the 
"first Weingarten family" deformation will change when different 
Weierstrass data is used, even when the original surface under 
consideration does not change.  

To demonstrate that other choices actually do give 
different deformations, we will also consider 
a deformation we call the 
"second Weingarten family", given by starting with a 
flat surface 
$f_0$ with given lift $E$, and then using the lift $EB$ for an 
off-diagonal $B \in \text{SU}_2$ instead to make the deformation 
through linear Weingarten surfaces of Bryant type.  
(Any choice of off-diagonal $B \in 
\text{SU}_2$ will result in the same deformation.)  Regardless of 
whether $E$ or $EB$ is used, we have the same surface $f_0$, but 
$E$ and $EB$ give opposite orientations for the normal vector 
to $f_0$, and we are interested in this particular choice for 
a second deformation precisely because of this 
orientation-reversing property.  The frames 
$E$ and $EB$ give different 
families of linear Weingarten surfaces when $t > 0$.  
Such different deformations of surfaces can be seen in 
Figures \ref{hypcmc} (first Weingarten family) and \ref{hypcmc2} 
(second Weingarten family), 
and also in Figures \ref{airycmc} (first Weingarten family) 
and \ref{airycmc2} (second Weingarten family).  

\subsection{The deformations with respect to the coordinate $w$}
The first Weingarten family is 
\[ f_t^{(1)} = (EL)(\overline{EL})^T \; , \] 
where $L$ is as in Section \ref{smoothLinWein} and $E$ solves Equation 
\eqref{E-eqn-for-flat-guys}.  
In terms of the new coordinate $w$ given in Section 
\ref{remark-on-q}, since $g = \int q dw$, $L$ takes the form 
\[ L = \begin{pmatrix}
\beta & -t \beta \cdot \int q dw \\ 0 & \beta^{-1}  
\end{pmatrix} \; , \;\;\; \beta = \beta(w) = 
\sqrt{\frac{1+t |\int q dw|^2}{1+t^2 |\int q dw|^2}} 
\; . \]

A second Weingarten family $f_t^{(2)}$ can be given by taking 
\[ B = \begin{pmatrix}
0 & 1 \\ -1 & 0
\end{pmatrix} \; , \] and then 
\[ f_t^{(2)} = E B \hat L (\overline{E B \hat L})^T \; , \] 
where 
\[ \hat L = \begin{pmatrix}
\hat \beta & -t h \hat \beta \\ 0 & \hat \beta^{-1}  
\end{pmatrix} \; , \;\;\; \hat \beta = 
\sqrt{\frac{1+t h \bar h}{1+t^2 h \bar h}} 
\; , \;\;\; h = \int \frac{-1}{g^\prime} dz \; , \]
and $E$ again solves \eqref{E-eqn-for-flat-guys}.  Note that there is 
freedom of choice of additive constant in the definition of $h$, 
and different constants will give different deformations.  

In terms of the new coordinate $w$ 
satisfying $dw = \frac{1}{g^\prime} dz$ (and 
$q = (g^\prime)^2$) given in 
Section \ref{remark-on-q}, a second Weingarten 
family $f^{(2)}_t$ has a particularly nice expression for 
its singular set:  

\begin{lemma}
The second Weingarten family $f^{(2)}_t$ taken by choosing 
$h = - w$ is singular along the curve
\[ |q|^2 (1+t |w|^2)^4-(1-t)^2 = 0 \; . \]  
\end{lemma}

Note that, in particular, $f^{(2)}_1$ could be singular 
only at points where $q=0$.  

\begin{proof}
Note that 
\[ (EB)^{-1} d(EB) = - \begin{pmatrix}
0 & 1 \\ q & 0 
\end{pmatrix} dw \] 
and 
\[ f_t^{(2)} = 
\begin{pmatrix}
x_0+x_3 & x_1+ i x_2 \\ x_1 - i x_2 & x_0-x_3 
\end{pmatrix} = 
E B 
\begin{pmatrix}
\tfrac{1+t^2|w|^2}{1+t|w|^2} & - t \bar w \\ 
-t w & 1+t|w|^2 
\end{pmatrix}
\bar B^T \bar E^T \; , 
\]
where the $x_j$ are now considered as real-valued 
functions of $w$ and $\bar w$ (and also of $t$).  
Substituting $dx_k = \tfrac{\partial x_k}{\partial w} dw + 
\tfrac{\partial x_k}{\partial \bar w} d\bar w$ 
into the Minkowski norm 
$dx_1^2+dx_2^2+dx_3^2-dx_0^2$, we find that 
\[ {\mathfrak A} dw^2+\bar {\mathfrak A} d\bar w^2+2 
{\mathfrak B} dwd\bar w = dx_1^2+dx_2^2+dx_3^2-dx_0^2 \]
has discriminant 
\[ {\mathfrak B}^2 - {\mathfrak A} \bar {\mathfrak A} = 
|u_1 \tfrac{d}{dw}(u_2)-u_2 \tfrac{d}{dw}(u_1)|^4 
\left( |q|^2 (1+t|w|^2)^4 - (1-t)^2 \right)^2 (1+t|w|^2)^{-4} \geq 0 
\; , \]
where $u_1$ and $u_2$ are as in Section \ref{remark-on-q}.  
Since $u_1 \tfrac{d}{dw}(u_2)-u_2 \tfrac{d}{dw}(u_1)=1$, the 
proof is completed.  
\end{proof}

\subsection{Examples}  We now give some examples.  

\begin{example}\label{examplewithqconstant}
Take any constant $q \in \mathbb{C} \setminus \{ 0 \}$ in Equation 
\eqref{eqn:q}.  Then we can take 
$g = \sqrt{q} z = q w$, and so both the coordinates 
$z$ and $w$ will be isothermic if $q \in \mathbb{R}$, 
and we now assume $q$ is a positive real.  Let us use the 
coordinate $w$, and then $E$ as in Section \ref{remark-on-q} can be 
taken as 
\[ E = q^{-1/4} \cdot \begin{pmatrix}
\cosh (\sqrt{q} w) & \sqrt{q} \sinh (\sqrt{q} w) \\ 
\sinh (\sqrt{q} w) & \sqrt{q} \cosh (\sqrt{q} w) 
\end{pmatrix}
\] and we can take $F$ as 
\[ F = q^{-1/4} \cdot \begin{pmatrix}
\cosh (\sqrt{q} w) & \sqrt{q} \sinh (\sqrt{q} w) - q w \cosh(\sqrt{q} w) \\ 
\sinh (\sqrt{q} w) & \sqrt{q} \cosh (\sqrt{q} w) - q w \sinh(\sqrt{q} w) 
\end{pmatrix} \; . \]
Then $f_0=E \bar E^T$ is a (geodesic) line when $|q|=1$, and is a 
(hyperbolic) cylinder when $|q| \neq 1$.  
Also, $f_1=F \bar F^T$ gives 
CMC $1$ Enneper cousins in $\mathbb{H}^3$ 
\cite{Br} \cite{RUY1} (see Figure \ref{hypcmc}).  These 
$f_1$ and $f_0$ deform to each other via the first Weingarten family.  

Now, consider half-lines in the domain (the complex plane $\mathbb{C}$) 
of the Enneper cousins $f_1$ emanating from the point $w=0$.  
In all but two directions for these rays, 
the corresponding curve on the surface will converge to a 
single point in 
the sphere at infinity $\partial \mathbb{H}^3$.  This limit point 
can be different for different directions, and in particular 
the limit will be one certain point (resp. one other certain point) 
for any ray that makes an angle of less than (resp. more than) 
$\pi/2$ with the positive real axis in the $w$-plane.  However, in 
the two special directions for these rays where $w$ is pure 
imaginary, the corresponding two curves on the 
surface converge to infinite wrappings of circles in 
$\partial \mathbb{H}^3$.  
These properties are easily checked, because we have 
an explicit form for $F$, as above.  Furthermore, the behavior is 
the same for any $f_t^{(1)}$ whenever $t > 0$, which is also easily 
checked.  This behavior is similar to the Stokes 
phenomenon we will see in Example \ref{thirdsmoothexample}.  

See Figure \ref{hypcmc} 
for the first Weingarten family of $f_0$, and Figure 
\ref{hypcmc2} for a second Weingarten family of $f_0$.  
\end{example}

\begin{example}\label{smoothsurfacesofrev} 
To make CMC $1$ surfaces of revolution $f_1$, the so-called 
catenoid cousins \cite{Br} \cite{UY1}, one can use 
$g=e^{\mu z}$ in \eqref{eqn:bryant}, for $\mu$ either real or purely 
imaginary.  For the corresponding flat 
surfaces $f_0$, one can obtain the surfaces called hourglasses 
(resp. snowmen) by 
choosing real (resp. imaginary) values for $\mu$ \cite{GMM-first} 
\cite{KUY}.  
Discrete versions of these flat surfaces can be seen in 
Figure \ref{uglyfigure8}, and graphics of 
the smooth surfaces look much the same, but are smooth.  
\end{example}

\begin{example}\label{thirdsmoothexample} 
We will now 
consider the holomorphic function $g = 
z^\gamma = z^{4/3}$, which gives the 
hyperbolic Schwarz map $f_0$ for the Airy equation $\tfrac{d^2}{dw^2} 
u - w u = 0$ as in Equation \eqref{eqn:q2}.  

The value $\gamma = \tfrac{4}{3}$ corresponds to the choice $q=w$, and 
so is of particular interest.  We can see this correspondence as follows: 
We take $q$ to be $w$.  This means we have $w = (g^\prime)^2$, 
and so $dw = 2 g^\prime g^{\prime\prime} dz$.  Then, because 
$dw=(1/(dg/dz)) dz$, we have 
$(1/(dg/dz)) dz = 2 g^\prime g^{\prime\prime} dz$, and so 
the original holomorphic function $g$, as a function of $z$, 
would satisfy 
\[ \frac{1}{2} = g^{\prime\prime} (g^\prime)^2 \] 
and so 
\[ g = \sqrt[3]{\tfrac{81}{128}} z^{4/3} \; , \]
and the scalar factor $\sqrt[3]{\tfrac{81}{128}}$ can be removed by 
replacing $z$ with an appropriate constant real multiple of $z$.  

This surface $f_0$ has an umbilic point at 
$w=0$ (so the corresponding caustic will 
blow out to infinity at $w=0$, see \cite{KRSUY}, \cite{KRUY1}, 
\cite{Roit}), and it has a 
"triangle" of singular points with three cuspidal edge arcs 
connected by three swallowtail 
singularities, and it has $120$ degree dihedral symmetry.  
Similar to the case of the CMC $1$ Enneper cousins, 
starting at the center point of the surface ($w=0$ in the domain of 
the mapping) and going in any direction but three (i.e. following a 
ray out from $w=0$ in the domain), the corresponding curve on the surface 
will converge to a single point (one of three possible points) 
in the sphere at infinity $\partial \mathbb{H}^3$ (as in Example 
\ref{examplewithqconstant}, 
this limit point can be different for different 
directions).  However, in three special directions for these rays, 
the corresponding curves on the surface converge to infinite 
wrappings of circles in $\partial \mathbb{H}^3$.  These special directions are 
exactly opposite to the directions of the swallowtail singularities.  
This is called {\em Stokes phenomenon}, and was explored carefully 
for this example in \cite{SY}, \cite{SYY}.  
Portions of this $f_0$ can be seen in Figure \ref{airy}.  

The surfaces $f_t^{(2)}$ ($0 \leq t < 1$) in a second Weingarten family 
of $f_0$ have the same property that the singular 
set is a triangle of three 
cuspidal edge arcs connected by three swallowtails.  Then, as $t$ tends 
to $1$, the three swallowtails converge to the origin $w=0$, and when 
$t=1$ the CMC $1$ surface $f_1$ has a branch point of order $2$ at 
$w=0$.  The deformation near the origin can be simulated by the 
map $Y_s : (u,v) \mapsto (y_1,y_2,y_3)$ given by 
\[ y_1 = - s (u^2+v^2) + 2 u v^2 - \tfrac{2}{3} u^3 \; , \;\;\; 
   y_2 = 2 s u + u^2 - v^2 \; , \;\;\; 
   y_3 = -s v + u v \; , \] 
as $s=1-t$ tends to zero.  

The first 
and second Weingarten families for $f_0$ can be seen in Figures 
\ref{airycmc} and \ref{airycmc2}.  

More generally, using 
$q=w^{n}$ in \eqref{eqn:bryant} will give the flat surfaces 
$f_0$ investigated in \cite{SY}.  
When $n$ is neither $-1$ nor $-2$, we have $g = c_n 
z^{\frac{2n+2}{n+2}}$ for some $c_n > 0$.  
When $q = w^{-1}$, then $g = 2 \log z$.  When $q=w^{-2}$, 
we have $g = c e^{-z}$, and $c$ can take on any value.  
\end{example}

\begin{remark}\label{smoothairyduality}
It is interesting to note that, in Example \ref{thirdsmoothexample}, 
while the two choices $q=w$ and $\hat q = \hat w^{-1/2}$ give different 
equations $\tfrac{d^2}{dw^2} u = w u$ and 
$\tfrac{d^2}{d\hat w^2} \hat u = \hat w^{-1/2} \hat u$, the solutions 
$\hat u$ to the second equation are essentially just 
$w$-derivatives of the solutions $u$ to the first equation.  A 
computation then shows $q$ and $\hat q$ can produce two flat surfaces 
that are parallel surfaces of each other.  In terms of the coordinate 
$z$, the corresponding statement is that $g = z^{4/3}$ and $\hat g = 
\hat z^{2/3}$ can produce two flat surfaces that are parallel to each 
other, and in fact this also follows from Lemma \ref{lem:3pt7} (taking 
$B$ to be off-diagonal with off-diagonal entries both $i$).  
\end{remark}

\begin{figure}[h]
\begin{center}
\begin{tabular}{ccc}
$t=0$ & $t=1/4$ & $t=1/2$  \\ 
\includegraphics[width=3cm]{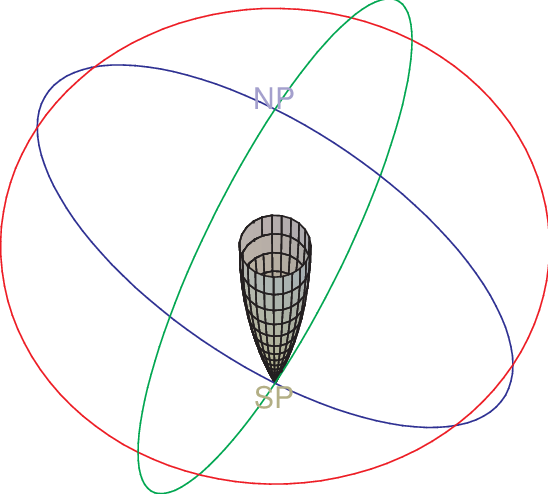} \qquad &
\qquad \includegraphics[width=3cm]{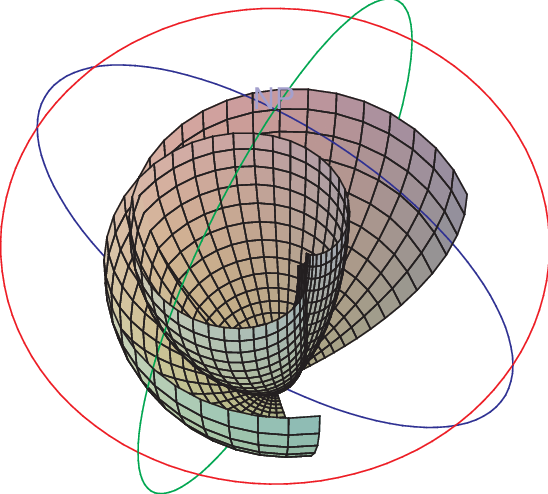} \qquad &
\qquad  \includegraphics[width=3cm]{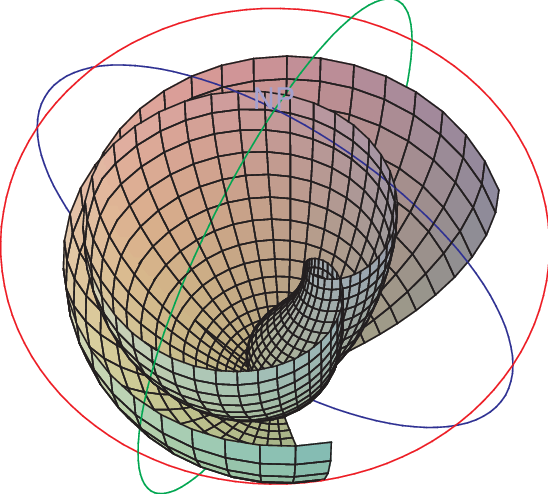} \\
$t=3/4$ & $t=1$ & \\
\includegraphics[width=3cm]{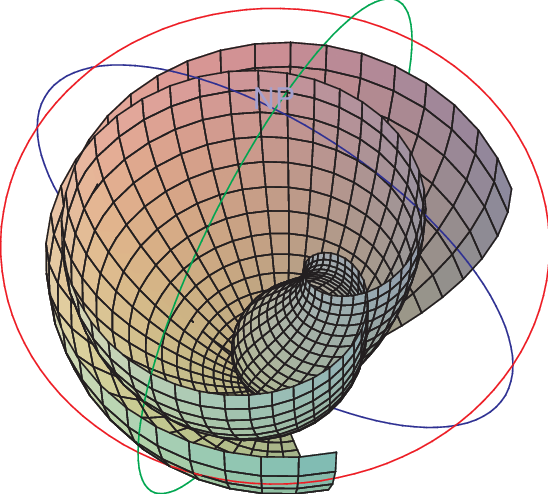}  \qquad & 
\qquad  \includegraphics[width=3cm]{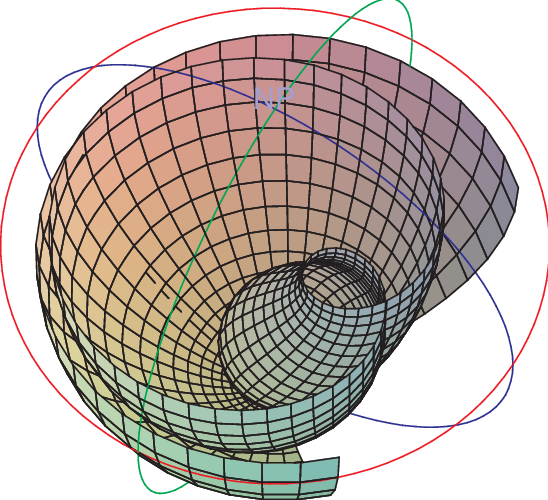}\qquad &
\end{tabular}
\end{center}
\caption{The first Weingarten family associated with a hyperbolic 
cylinder, in the Poincare ball model for $\mathbb{H}^3$.  
Here we take $q>0$ constant, so that $g^\prime = \sqrt{q}$ and 
$g = q w$.  Half of each of the surfaces is cut away.  Note that the 
surface is a CMC $1$ Enneper cousin when $t=1$.  See Example 
\ref{examplewithqconstant}.  (Here we have included grid lines 
on these smooth surfaces simply to make the graphics more visible.)}
\label{hypcmc}
\end{figure}

\begin{figure}[h]
\begin{center}
\begin{tabular}{ccc}
$t=0$ & $t=1/4$ & $t=1/2$  \\ 
\includegraphics[width=3cm]{cyll091m.eps} \qquad &
\qquad \includegraphics[width=3cm]{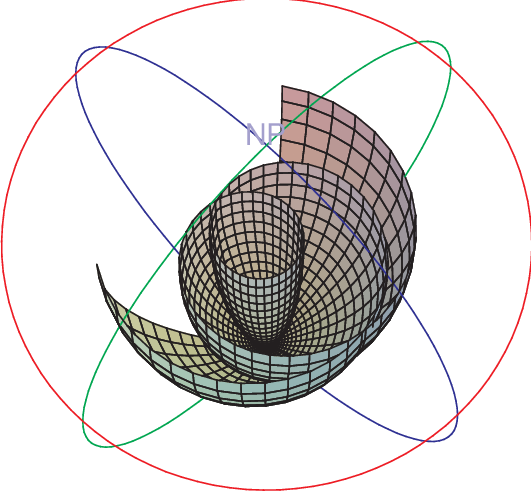} \qquad &
\qquad  \includegraphics[width=3cm]{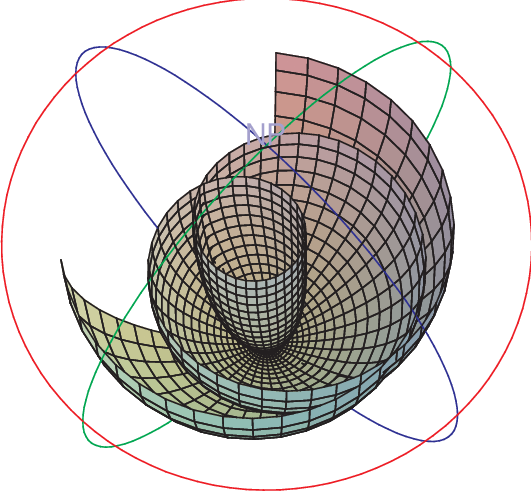} \\
$t=3/4$ & $t=1$ & \\
\includegraphics[width=3cm]{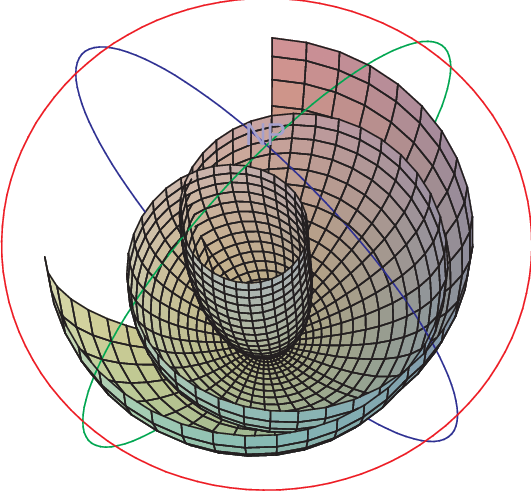}  \qquad & 
\qquad  \includegraphics[width=3cm]{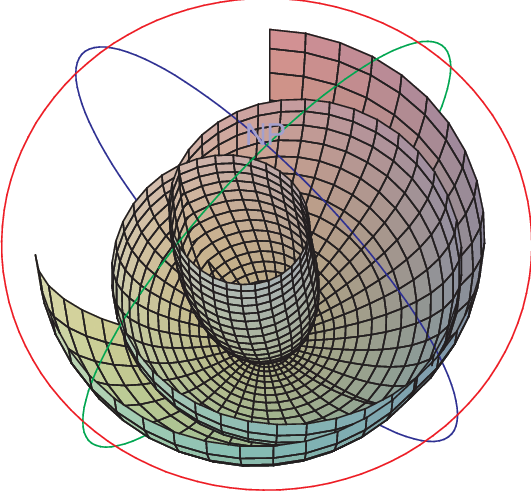}\qquad &
\end{tabular}
\end{center}
\caption{A second Weingarten family associated with a hyperbolic cylinder, 
in the Poincare ball model.  See Example \ref{examplewithqconstant}.} 
\label{hypcmc2}
\end{figure}

\begin{figure}[h]
\begin{center}
  \includegraphics[width=2cm]{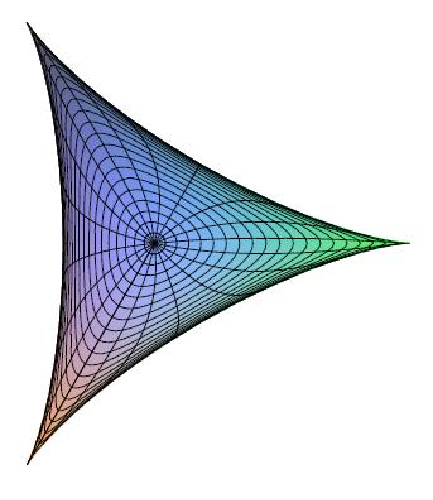}
\qquad
  \includegraphics[width=3.2cm]{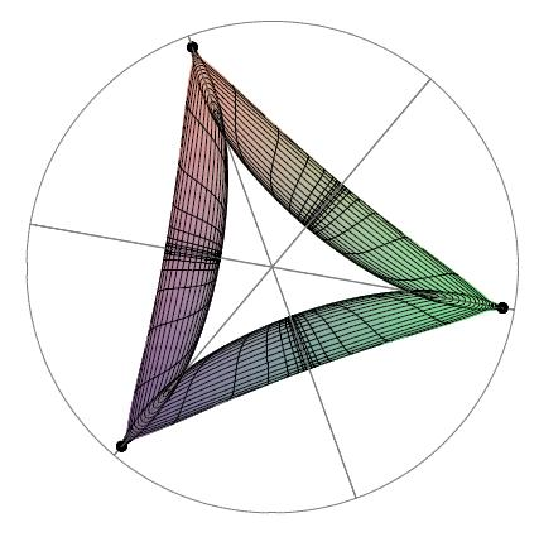}
\qquad
  \includegraphics[width=3.2cm]{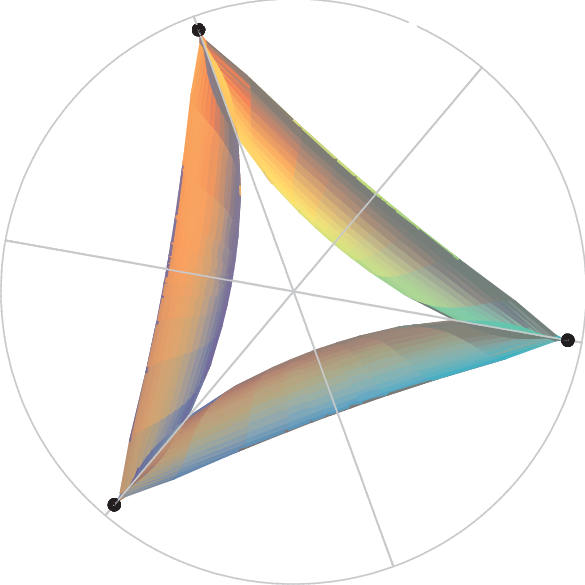}
\qquad
  \includegraphics[width=3.2cm]{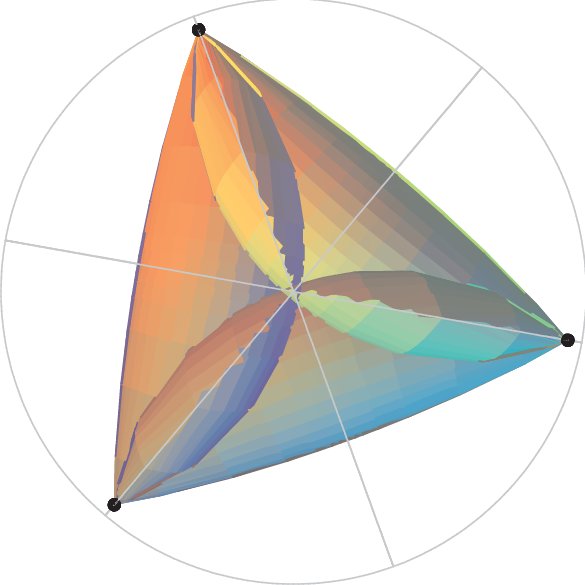}
\end{center}
\caption{Images of the smooth flat surface related to the 
Airy equation $\frac{d^2}{dw^2} u - w u = 0$, in the 
Poincare ball model.  The left figure is the 
image of the unit disc $\{ |w| \leq 1 \}$, shown with grid lines 
for better visibility.  The boundary of this left figure is the 
cuspidal edge curve with three swallowtail corners.  The middle  
two figures are the image of the ring 
$\{ 1.17 \leq w \leq 2.34 \}$, shown twice, once with grid lines 
and once without.  The right figure is the image of 
$\{ 3.27 \leq w \leq 4.09 \}$, this time shown without grid lines 
because in this case visibility is better without them.  See Example 
\ref{thirdsmoothexample}.} 
\label{airy}
\end{figure}

\begin{figure}[h]
\begin{center}
\begin{tabular}{cc}
$t=0$\quad (flat) & $t=1/2$  \\
  \includegraphics[width=4cm]{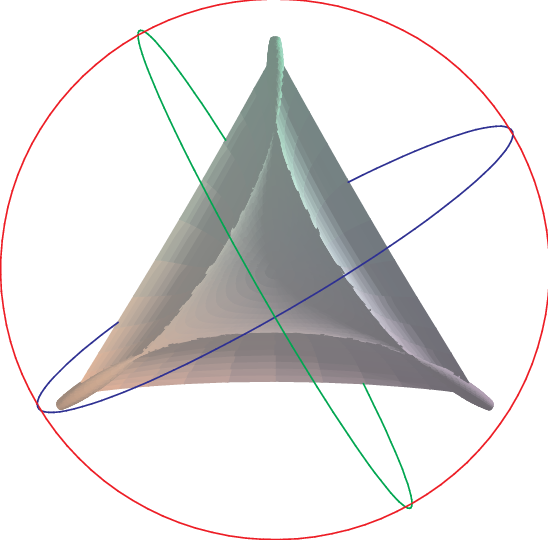} \qquad &
\qquad  \includegraphics[width=4cm]{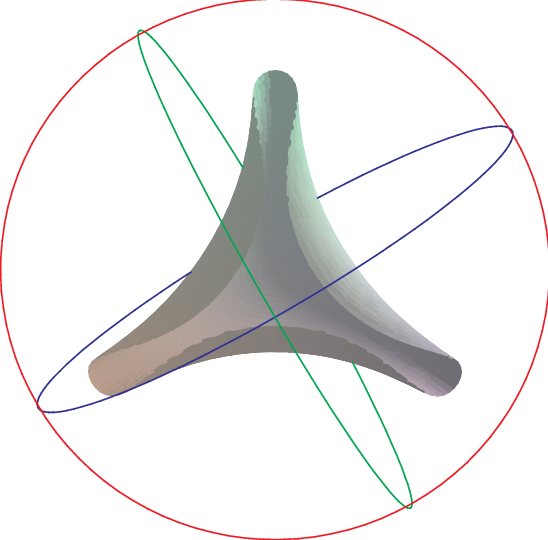} \\
$t=1$\quad (CMC 1) & $t=1$\quad (CMC 1) \\
  \includegraphics[width=4cm]{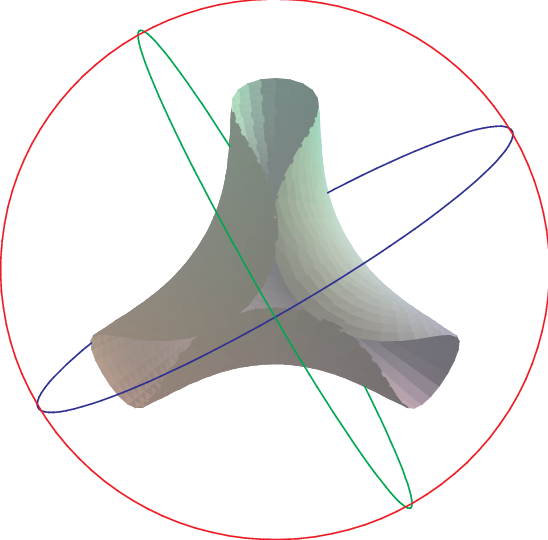} \qquad &
\qquad   \includegraphics[width=4cm]{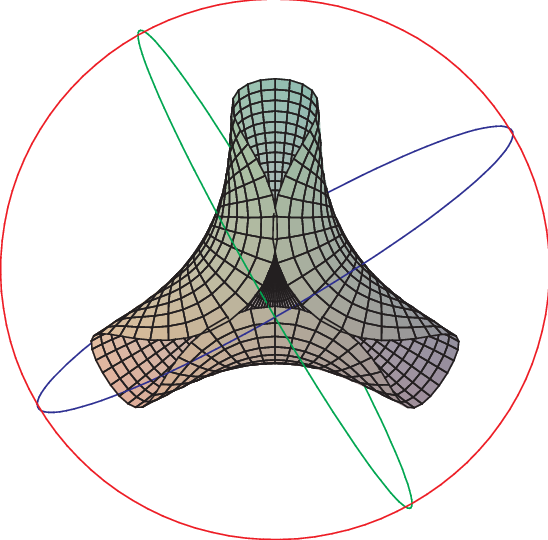} 
\end{tabular}
\end{center}
\caption{The first Weingarten family associated with the Airy equation, 
in the Poincare ball model.  See Example 
\ref{thirdsmoothexample}.  The bottom two figures are the same 
CMC $1$ surface, shown once with grid lines and once without.} 
\label{airycmc}
\end{figure}

\begin{figure}[h]
\begin{center}
\begin{tabular}{ccc}
$t=0$\quad (flat) & $t=1/10$  & \\
  \includegraphics[width=4cm]{airy091n.eps} \qquad &
\qquad  \includegraphics[width=4cm]{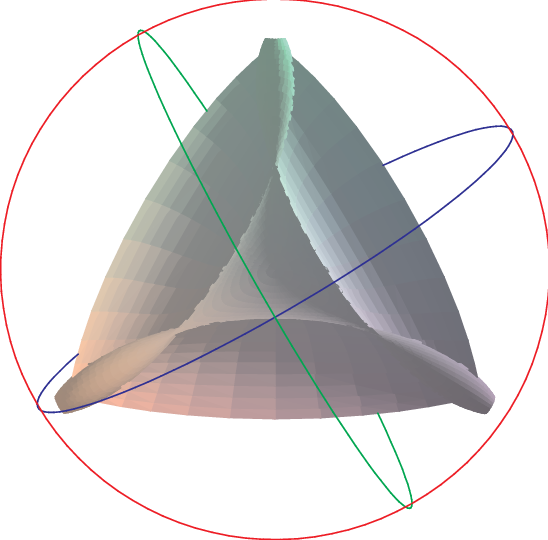} & \\
$t=1/2$ & $t=1$\quad (CMC 1) & $t=1$\quad (CMC 1) \\
  \includegraphics[width=4cm]{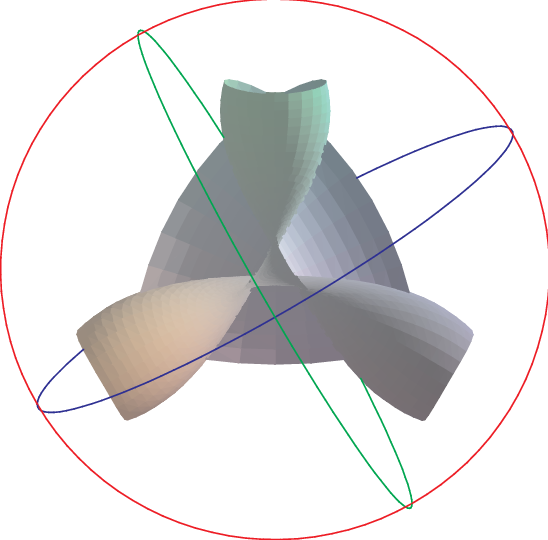} \qquad &
\qquad   \includegraphics[width=4cm]{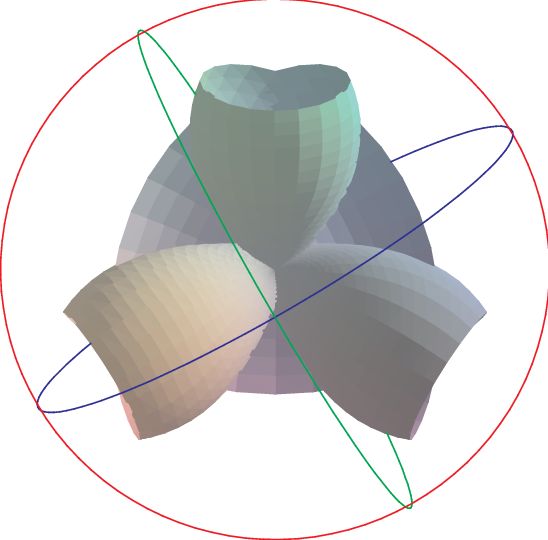} \qquad &
\qquad   \includegraphics[width=4cm]{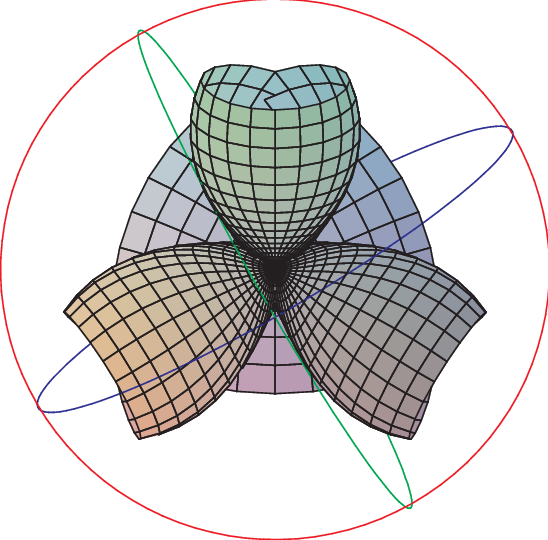} 
\end{tabular}
\end{center}
\caption{A second Weingarten family associated with the Airy equation, 
in the Poincare ball model.  See Example 
\ref{thirdsmoothexample}.  The two lower-right figures are the same 
CMC $1$ surface, shown once with grid lines and once without.  
(The upper-left surface here is the same as the upper-left surface 
in Figure \ref{airycmc}.)}  
\label{airycmc2}
\end{figure}

\section{Discrete CMC $1$, flat and linear Weingarten surfaces 
in $\mathbb{H}^3$}
\label{section3}

\subsection{Known definition for discrete CMC $1$ surfaces}
We now describe how discrete CMC $1$ surfaces 
in $\mathbb{H}^3$ were defined in \cite{Udo1}.  For this, we first give 
a light cone model for $\mathbb{H}^3$ in Minkowski $5$-space $\mathbb{R}^{4,1}$ 
that is commonly used in Moebius geometry, and 
that was used in \cite{Udo1} in conjunction with 
quaternions.  Let $\mathcal{H}$ denote the $4$-dimensional vector 
space of quaternions with the usual basis $1$, $i$, $j$ and $k$, and 
with the usual notion of quaternionic multiplication.  Points 
$(x_1,x_2,x_3,x_4,x_0) \in \mathbb{R}^{4,1}$ can then be given by 
\begin{equation}\label{mink5sp} 
X = x_1 \begin{pmatrix} i & 0 \\ 0 & -i \end{pmatrix} + 
x_2 \begin{pmatrix} j & 0 \\ 0 & -j \end{pmatrix} + 
x_3 \begin{pmatrix} k & 0 \\ 0 & -k \end{pmatrix} + 
x_4 \begin{pmatrix} 0 & 1 \\ -1 & 0 \end{pmatrix} + 
x_0 \begin{pmatrix} 0 & 1 \\ 1 & 0 \end{pmatrix} \; , 
\end{equation}
and the $\mathbb{R}^{4,1}$ metric $\langle X,X \rangle$ is then given by 
($I$ is the identity matrix) 
\[ \langle X,X \rangle \cdot I = -X^2 \; . \]  
Now let us view the collection of such trace-free 
$2 \times 2$ matrices with imaginary quaternions on the diagonal 
and reals on the off-diagonal as the set of points in 
$\mathbb{R}^{4,1}$.  
We can then define $\mathbb{H}^3$ as a $3$-dimensional submanifold of the 
light cone in this way:  
\begin{equation}\label{H3quatmodel} 
\mathbb{H}^3 = \left\{ X \in \mathbb{R}^{4,1} \, \left| \, X^2 = 0 , 
      X \cdot \begin{pmatrix} -i & 0 \\ 0 & i \end{pmatrix} + 
      \begin{pmatrix} -i & 0 \\ 0 & i \end{pmatrix} \cdot X = 
      2 I \right. \right\} \; . \end{equation} 
This $\mathbb{H}^3$ will have constant sectional curvature $-1$ when 
given the induced metric from $\mathbb{R}^{4,1}$.  
Furthermore, for each point $X$ in the light cone, there is at most 
one value for $r \in \mathbb{R}$ so that $rX$ lies in $\mathbb{H}^3$, 
so we can 
alternately view $\mathbb{H}^3$ as the projectivized light cone.  

In \cite{Udo1}, in order to define discrete CMC $1$ surfaces 
in $\mathbb{H}^3$, the discrete version of the Bryant equation 
\begin{equation}\label{discretebryeqn} 
F_q-F_p = F_p \begin{pmatrix} g_p & -g_pg_q \\ 1 & -g_q \end{pmatrix} 
\frac{\lambda \alpha_{pq}}{g_q-g_p} \; , \;\;\; 
\det F \in \mathbb{R} \; , \end{equation} was used, where 
$g$ is a discrete holomorphic function with 
cross ratio factorizing function $\alpha_{pq}$, and $p$ and $q$ are adjacent 
vertices in the domain $D \subset \mathbb{Z}^2$ of $g$.  
Note that we have assumed $g_q-g_p \neq 0$ (see Section 
\ref{sect:2pt3}).  The nonzero real parameter 
$\lambda$ can be chosen freely.  The formula for the discrete CMC 
$1$ surface $f$ in $\mathbb{H}^3$ was then obtained, analogous to the 
formula \eqref{starstar} for the case of 
discrete minimal surfaces, by setting (here $\mathbb{L}^4$ is the $4$-dimensional 
light cone in $\mathbb{R}^{4,1}$) 
\begin{equation}\label{eqn:form-we-want} f_{1,p} = r \cdot 
\begin{pmatrix} -b \bar a & a \bar a \\ b \bar b & - a \bar b \end{pmatrix} 
\in \mathbb{H}^3 \subset \mathbb{L}^4 
\subset \mathbb{R}^{4,1} \; , \;\;\; \text{where} \;\; 
\begin{pmatrix} a \\ b \end{pmatrix} = \begin{pmatrix} 0 & 1 \\ 
j & 0 \end{pmatrix} 
F_p \begin{pmatrix} i \\ j \end{pmatrix} \; , 
\end{equation} where $r$ is the appropriate choice of real scalar to 
place $f_{1,p}$ in $\mathbb{H}^3$ as defined in \eqref{H3quatmodel}.  In fact, 
\[ r = 2 (b \bar a i + i b \bar a)^{-1} = -2 (a \bar b i + i 
a \bar b)^{-1} \in \mathbb{R} \; . \]  Because the entries of $F$ are complex, 
not quaternionic, it follows that 
$b \bar a$ is purely imaginary quaternionic, so the above 
matrix $f_{1,p}$ does lie in $\mathbb{R}^{4,1}$, and thus in 
$\mathbb{L}^4$ and then 
also in $\mathbb{H}^3$ with appropriate choice of $r$.  Also, we 
actually have a $1$-parameter family of surfaces, due to the freedom of 
choice of $\lambda$.  

The solution $F$ is defined only up to scalar factors.  
This is essentially because Equation \eqref{discretebryeqn} is not 
symmetric in $p$ and $q$, and can be explained as follows: 
Consider a quadrilateral in $D$ with vertices $p$, $q$, $r$ and $s$ 
given counterclockwise about the quadrilateral, and with $p$ as the 
lower left vertex.  Then $F_q = F_p \cdot \mathfrak{A}$ with 
$\mathfrak{A}$ determined by Equation \eqref{discretebryeqn}.  
Then $F_r = F_q \cdot \mathfrak{B}$, again by \eqref{discretebryeqn}.  
Similarly, $F_s = F_p \cdot \mathfrak{C}$ and $F_r = F_s \cdot 
\mathfrak{D}$.  Thus we expect $\mathfrak{A} 
\mathfrak{B} = \mathfrak{C} \mathfrak{D}$, and a computation shows that 
this is indeed the case.  However, we also have $F_p$ equal to 
$F_q \cdot \hat{\mathfrak{A}}$, by \eqref{discretebryeqn} with the roles 
of $p$ and $q$ reversed, and it turns out that $\mathfrak{A} 
\hat{\mathfrak{A}} = (1 - 
\lambda \alpha_{pq}) I \neq I$.  Thus, the solution 
$F$ to \eqref{discretebryeqn} is only defined projectively, that is, 
it is only defined up to scalar factors.  

Nevertheless, because we 
are scaling by a real factor $r$ in Equation \eqref{eqn:form-we-want} 
anyways, the resulting discrete CMC $1$ surface is still well-defined.  
Thus we have seen the following lemma.  

\begin{lemma}\label{welldefinedCMC1}
Although the solution $F$ of Equation \eqref{discretebryeqn} 
is multi-valued and only 
defined up to scalar factors, the discrete CMC $1$ surface $f_{1}$ given 
in \eqref{eqn:form-we-want} is well-defined.  
\end{lemma}

In fact, the upcoming Theorem \ref{nice-form-for-f} 
also implies that the discrete CMC $1$ surface $f_{1}$ is well-defined.

\subsection{New formulation for discrete CMC $1$ surfaces}
Equivalently to the definition given in \cite{Udo1}, there is 
another way to define $f_p$, which is given in the next theorem.  This 
new form for $f_p$ is convenient, because 
it is clearly analogous to the form used in the Bryant 
representation for smooth CMC $1$ surfaces in $\mathbb{H}^3$, and it removes the 
need for the Moebius-geometric $\mathbb{R}^{4,1}$ light cone model for 
$\mathbb{H}^3$.  

\begin{theorem}\label{nice-form-for-f}
The above description \eqref{eqn:form-we-want} 
for the discrete CMC $1$ surface given by 
$F$ is equal to the surface given by $\frac{1}{\det F} F \bar F^T$, 
up to a rigid motion of $\mathbb{H}^3$.  
\end{theorem}

\begin{proof}
The matrices $f_p$ in $\mathbb{H}^3$, as described in \eqref{eqn:form-we-want}, 
will be of the form 
\[ r \begin{pmatrix} -(\bar A C+\bar B D)j+i(AD-BC) & C\bar C+D \bar D \\
A \bar A + B \bar B & j (A \bar C + B \bar D)-i(AD-BC)
\end{pmatrix} \; , 
\] where $r$ is a nonzero real scalar and 
\[ F = \begin{pmatrix} A & B \\ C & D 
\end{pmatrix} \; . 
\] To have $f_{1,p} \in \mathbb{H}^3$, we should take 
\[ r = \frac{1}{AD-BC} \; . \]  This means that the coefficient of the 
$i$ term in the diagonal entries will be simply $\pm 1$.  So we can view 
the surface as lying in the $4$-dimensional space $\mathbb{R}^{3,1}$, by simply 
removing the matrix term with scalar $x_1$ from Equation \eqref{mink5sp}.  

Now, the projection into the Poincare ball model is 
\[ (x_2,x_3,x_4,x_0) \to \frac{(x_2,x_3,x_4)}{1+x_0} = \]
\begin{equation}\label{firstPoincareprojection} 
\frac{(\text{Re}(-\bar A C-\bar B D),\text{Im}(-\bar A C-\bar B D),
\tfrac{1}{2}(-A\bar A-B\bar B+C\bar C+D\bar D))}{AD-BC + 
\tfrac{1}{2} (A\bar A+B\bar B+C\bar C+D\bar D)} \; . \end{equation} 

On the other hand, if we look at 
\[ \frac{1}{AD-BC} F \bar F^T = \frac{1}{AD-BC} 
\begin{pmatrix}
A \bar A+B \bar B & A \bar C+B \bar D \\ C \bar A+D \bar B & C \bar C+D \bar D
\end{pmatrix} = 
\begin{pmatrix}
y_0+y_3 & y_1+i y_2 \\ y_1-i y_2 & y_0-y_3
\end{pmatrix} \; ,\] and then project to the Poincare ball, we 
have 
\[ \frac{(y_1,y_2,y_3)}{1+y_0} = \]
\begin{equation}\label{secondPoincareprojection} 
\frac{(\text{Re}(A \bar C+B \bar D),\text{Im}(A \bar C+B \bar D),
\tfrac{1}{2}(A\bar A+B\bar B-C\bar C-D\bar D))}{AD-BC + 
\tfrac{1}{2} (A\bar A+B\bar B+C\bar C+D\bar D)} \; . \end{equation} 
The quantities \eqref{firstPoincareprojection} and 
\eqref{secondPoincareprojection} are the same, up to a rigid motion of 
$\mathbb{H}^3$, proving the theorem.  
\end{proof}

\subsection{Discrete flat surfaces}\label{section4pt3}
To make discrete flat surfaces in $\mathbb{H}^3$, we can now take 
\begin{equation}\label{EpEpEp} 
E_p = F_p \cdot \begin{pmatrix} 1 & g_p \\ 0 & 1 \end{pmatrix} \; , 
\;\;\; \det E = \det F \in \mathbb{R} \; , 
\end{equation}  
like in \eqref{EinrelationtoF} for the smooth case.  
We then use the same formula \eqref{eqn:form-we-want} as for discrete 
CMC $1$ surfaces to define the discrete flat surface, but with $F_p$ 
replaced by $E_p$.  In light of Theorem \ref{nice-form-for-f}, 
the following definition is natural: 

\begin{definition}
For $E$ given as in \eqref{EpEpEp}, where $F$ is a solution of 
Equation \eqref{discretebryeqn}, 
\begin{equation}\label{starstar1} 
f_{0} = \frac{1}{\det E} E \bar E^T \end{equation} is a {\em discrete flat} 
surface.  
\end{definition}

Furthermore, in light of the behavior of the normal for smooth 
flat surfaces in Equation \eqref{flatsurfnormal}, it is natural 
to define the normal at vertices of a discrete flat surface by 
\begin{equation}\label{flatsurfnormal2} 
N_p := \frac{1}{\det E_p} E_p \cdot \begin{pmatrix} 
1 & 0 \\ 0 & -1 \end{pmatrix} \cdot \bar E_p^T \; . \end{equation}
A discrete version of Equation \eqref{E-eqn-for-flat-guys} can then 
be computed, and becomes 
\begin{equation}\label{star15} 
E_q-E_p = E_p \begin{pmatrix} 0 & g_q-g_p \\ \frac{\lambda 
\alpha_{pq}}{g_q-g_p} & 0 \end{pmatrix} \; , \end{equation} 
where $\lambda$ is an arbitrary parameter in $\mathbb{R} \setminus 
\{ 0 \}$.  

\begin{remark}
Like in Lemma \ref{welldefinedCMC1} for $F$, this $E$ is 
only defined up to scalar factors.  However, the discrete flat 
surface $f_0$ is still well-defined.  
\end{remark}

\begin{remark}\label{rem-parallelsurfaces} 
Changing $g_p$ to $d \cdot g_p$ for $d \in \mathbb{R}^+$ gives parallel 
discrete flat surfaces, and this is seen as follows: 
$g_p \to d \cdot g_p$ implies 
\begin{equation}\label{starstar2} 
E_p \to E_p^d := E_p \cdot \begin{pmatrix} 1/\sqrt{d} & 0 \\ 0 & \sqrt{d} 
\end{pmatrix} \; , \end{equation}  
which implies the original surface $f_{0,p} = \tfrac{1}{\det E_p} 
E_p \bar E_p^T$ changes to 
\begin{equation}\label{starstar3} 
f_{0,p}^d = \cosh (\log d) \, \cdot f_{0,p} - \sinh (\log d) 
\, \cdot N_p \; . \end{equation}  This surface $f_{0}^d$ is a parallel 
surface of $f_{0}$, and $f_{0}^1$ is the same as the original surface 
$f_{0}$.  It also follows from \eqref{starstar2} and \eqref{starstar3} 
that the geodesic at $f_{0,p}$ in the direction $N_p$ and the geodesic at 
$f_{0,p}^d$ in the direction $N_p^d = \tfrac{1}{\det E_p^d} E_p^d 
\begin{pmatrix} 1 & 0 \\ 0 & -1 \end{pmatrix} \overline{E_p^d}^T$ 
are the same.  
\end{remark}

We now give one of our main results: 

\begin{theorem}\label{thm:concirc-quads}
Discrete flat surfaces in $\mathbb{H}^3$ have concircular quadrilaterals.  
\end{theorem}

\begin{remark}\label{rem:nonisothermicity}
Although they have concircular quadrilaterals, discrete flat surfaces 
are generally not discrete isothermic.  This is expected, since smooth 
flat surfaces given as in Section \ref{section3pt2} are 
generally not isothermically parametrized as well.  
\end{remark}

\begin{proof}
Let $f_{0} = \frac1{\det(E)}E \bar{E}^T$ be a discrete flat surface 
defined on a domain 
$D \subseteq \mathbb{Z}^2$, formed from a discrete holomorphic function $g$ on $D$.  
Let $p=(m,n)$, $q=(m+1,n)$, $r=(m+1,n+1)$ and $s=(m,n+1)$ be the vertices 
of one quadrilateral in $D$, and let $E_p$, $E_q$, $E_r$ and $E_s$ be 
the respective values of $E$ at those vertices.  
We choose a cross ratio factorizing function for $g$, and we denote that 
function's value on the edge from $p$ to $q$, 
respectively $p$ to $s$, as $\alpha_{pq}$, respectively 
$\alpha_{ps}$.  Then 
$$
\begin{array}{rcl}
  E_q = E_p U \; , \;\;\;\;\; 
  U &=& \quadmatrix{1}{g_q-g_p}{\frac{\lambda \alpha_{pq}}{g_q-g_p}}{1} \; , \\
  E_s = E_p V \; , \;\;\;\;\; 
  V &=& \quadmatrix{1}{g_s-g_p}{\frac{\lambda \alpha_{ps}}{g_s-g_p}}{1} \; , \\
  E_r = E_p U V_1 = E_p V U_1 \; , \;\;\;\;\; 
  U_1 &=& \quadmatrix{1}{g_r-g_s}{\frac{\lambda \alpha_{pq}}{g_r-g_s}}{1} 
  \; , \\
  V_1 &=& \quadmatrix{1}{g_r-g_q}{\frac{\lambda \alpha_{ps}}{g_r-g_q}}{1} 
\end{array}
$$ 
(note that we have $UV_1 = VU_1$, by the cross ratio identity in 
Equation \eqref{rats2}).  

Now the cross ratio for the
quadrilateral given by the four surface vertices $f_{0,p}$, $f_{0,q}$, 
$f_{0,r}$ and $f_{0,s}$ is 
\[ C = (f_{0,p} - f_{0,q})(f_{0,q} - f_{0,r})^{-1} (f_{0,r} - 
f_{0,s})(f_{0,s} - f_{0,p})^{-1} \; , \] 
or equivalently \[ (f_{0,p} - f_{0,q})(f_{0,q} - f_{0,r})^{-1} = 
C (f_{0,p} - f_{0,s})(f_{0,s} - f_{0,r})^{-1} \; . \]  
(Note that commutativity does not hold for the product 
of four terms in this $C$, so it is vital that the 
order of the product be given correctly.)  
To prove the theorem, it suffices to show that $C$ is a real scalar factor 
times the identity matrix.  

Note that 
\[ f_{0,p} - f_{0,q} = 
\frac1{\det(E_p)} E_p ( I - \frac1{\det(U)}U \bar U^T) \bar E_p^T \]
and
\[ 
f_{0,q} - f_{0,r} = 
\frac1{\det(E_p)} \frac1{\det(U)} E_p U ( I - \frac1{\det(V_1)} V_1 
\bar V_1^T) \bar U^T \bar E_p^T \; , \]
so we have
\[ 
(f_{0,p} - f_{0,q})(f_{0,q} - f_{0,r})^{-1} = \det(U) E_p 
( I - \frac1{\det(U)}U \bar U^T) (\bar U^T)^{-1} ( I - \frac1{\det(V_1)} V_1 
\bar V_1^T)^{-1} U^{-1} E_p^{-1} \; , \]
and likewise
\[
(f_{0,p} - f_{0,s})(f_{0,s} - f_{0,r})^{-1} = 
\det(V) E_p ( I - \frac1{\det(V)}V \bar V^T) 
(\bar V^T)^{-1}  ( I - \frac1{\det(U_1)} U_1 \bar U_1^T)^{-1} V^{-1} 
E_p^{-1} \; . \]
So one finds that 
\[
\det(U_1) E_p (\det(U) (\bar U^T)^{-1} - 
U)(\det(V_1)V_1^{-1}- \bar V_1^T)^{-1} V_1^{-1} U^{-1} =
\]
\[
 = \det(V_1) C E_p (\det(V)(\bar V^T)^{-1} - V)(\det(U_1) U_1^{-1} - 
\bar U_1^T)^{-1} U_1^{-1} V^{-1} \; . \]
Because $U V_1 = V U_1$, we have 
\[
\frac{\det(U_1)}{\det(V_1)} E_p (\det(U) (\bar U^T)^{-1} - 
U)(\det(V_1)V_1^{-1} - \bar V_1^T)^{-1} = 
\] 
\[ C E_p (\det(V)(\bar V^T)^{-1} - V)(\det(U_1) U_1^{-1} - 
\bar U_1^T)^{-1} \; . 
\]
The determinants of $U$ and $V$ are real, so for example 
$ \det(U)(\bar U^T)^{-1} = 
\quadmatrix{1}{\frac{\lambda \alpha_{pq}}{\bar g_p - 
\bar g_q}}{\bar g_p - \bar g_q}{1}$ and we get 
\[
 \det(U) (\bar U^T)^{-1} - 
U = (|g_q-g_p|^2+\lambda \alpha_{pq}) 
\begin{pmatrix}
0 & \frac{1}{\bar g_p-\bar g_q} \\ 
\frac{1}{g_p-g_q} & 0 
\end{pmatrix} \; .
\]
With similar expressions for the other differences we see that 
\[
\frac{\det(U_1)}{\det(V_1)} 
\frac{|g_q-g_p|^2+\lambda \alpha_{pq}}{|g_r-g_q|^2+\lambda \alpha_{ps}}
E_p 
\begin{pmatrix}
0 & \frac{1}{\bar g_p - \bar g_q} \\ 
\frac{1}{g_p - g_q} & 0 
\end{pmatrix}
\begin{pmatrix}
0 & g_q - g_r \\ 
\bar g_q - \bar g_r & 0 
\end{pmatrix}
= \]
\[ =  
\frac{|g_s-g_p|^2+\lambda \alpha_{ps}}{|g_r-g_s|^2+\lambda \alpha_{pq}} 
C E_p 
\begin{pmatrix}
0 & \frac{1}{\bar g_p - \bar g_s} \\ 
\frac{1}{g_p - g_s} & 0 
\end{pmatrix}
\begin{pmatrix}
0 & g_s - g_r \\ 
\bar g_s - \bar g_r & 0 
\end{pmatrix} \; . 
\]
So the expression for $C$ is
\[ C = \mu E_p 
\begin{pmatrix}
\frac{\bar g_q - \bar g_r}{\bar g_p - \bar g_q}\frac{\bar g_s - 
\bar g_p}{\bar g_r - \bar g_s} & 0 \\ 
0 & \frac{g_q - g_r}{g_p - g_q}\frac{g_s -
g_p}{g_r -g_s}
\end{pmatrix} 
E_p^{-1} = \mu \cdot \frac{\alpha_{ps}}{\alpha_{pq}} 
\cdot \quadmatrix{1}{0}{0}{1} \]
for some real factor $\mu$.  This concludes the proof.  
\end{proof}

\subsection{Discrete linear Weingarten surfaces of Bryant type}
\label{section:dlwsobt}
We can now look at discrete linear 
Weingarten surfaces in $\mathbb{H}^3$, similarly to the approach taken 
in Subsection \ref{smoothLinWein} 
for the smooth case.  In the discrete case, one takes 
$E_p(t) = E_p \cdot L_p$, for 
\[ L_p = \begin{pmatrix}
\sqrt{\frac{1 + t g_p\bar g_p}{1 + t^2 g_p\bar g_p}} & 
- t g \sqrt{\frac{1 + t g_p\bar g_p}{1 + t^2 g_p\bar g_p}} \\ 
0 & \sqrt{\frac{1 + t^2 g_p\bar g_p}{1 + t g_p\bar g_p}} 
\end{pmatrix} \; . \]
We can then define the discrete linear Weingarten surface of 
Bryant type to be 
\[ f_{t,p} = \frac{1}{\det E_p(t)} E_p(t) \overline{E_p(t)}^T \; . \]  

Note that even though $f_t$ has been defined using one particular 
choice for the deformation through linear Weingarten surfaces 
(and as seen before, this choice is not canonical), the 
resulting collection of all discrete linear Weingarten surfaces does 
not depend on the choice of deformation.  This follows from properties 
analogous to those for the smooth case in 
Section \ref{rem:ambiguityofGMM}: 
\begin{itemize}
\item If $g_{m,n}$ is a discrete holomorphic function defined 
on a domain $D$, then $g_{m,n}+a$ is also discrete holomorphic 
with the same cross ratios, for any choice of complex 
constant $a$.  
\item If a cross ratio factorizing function for $g_{m,n}$ is 
$\alpha$, then $\hat g_{m,n}$ defined by 
\begin{equation}\label{eqn:geometryandsomething} 
\hat g_q-\hat g_p = \frac{-\alpha_{pq}}{g_q-g_p} \end{equation} 
(where $pq$ represents both horizontal and vertical edges) 
is also a discrete holomorphic function that is well defined 
on $D$ (i.e. $\hat g_{m,n}$ is not multi-valued 
once an initial condition 
is fixed in the above difference equation 
\eqref{eqn:geometryandsomething}), again with the 
same cross ratios as $g_{m,n}$.  
\item If Equation \eqref{star15} holds, then we also have 
\[ 
(E_q-E_p) \cdot \begin{pmatrix}
0 & \tfrac{1}{\sqrt{\lambda}} \\ -\sqrt{\lambda} & 0 
\end{pmatrix} = 
E_p \cdot \begin{pmatrix}
0 & \tfrac{1}{\sqrt{\lambda}} \\ -\sqrt{\lambda} & 0 
\end{pmatrix} \cdot 
\begin{pmatrix}
0 & \hat g_q-\hat g_p \\ 
\tfrac{\lambda \alpha_{pq}}{\hat g_q-\hat g_p} & 0 
\end{pmatrix} \; , 
\] and so both $E$ and 
\[ E \cdot \begin{pmatrix}
0 & \tfrac{1}{\sqrt{\lambda}} \\ -\sqrt{\lambda} & 0 
\end{pmatrix}
\] will produce discrete flat surfaces in $\mathbb{H}^3$, and these 
two surfaces are parallel surfaces of each other (see 
Remark \ref{rem-parallelsurfaces}).  
The two resulting linear Weingarten families will be different.  
\item For any constant matrix 
\[ \begin{pmatrix} 
a & b \\ -\bar b & \bar a 
\end{pmatrix} \in \text{SU}_2 \; , \] 
The function 
\[ \tilde g_{m,n} = \frac{a g_{m,n} + b}{-\bar b g_{m,n} + \bar a} 
\] is also discrete holomorphic, with the same cross ratios as $g$.  
\item If Equation \eqref{discretebryeqn} holds, then we also have 
\[ (F_q-F_p) \cdot \begin{pmatrix} 
\bar a & - b \\ \bar b & a 
\end{pmatrix} = F_p \cdot \begin{pmatrix}
\bar a & - b \\ \bar b & a 
\end{pmatrix} \cdot 
\begin{pmatrix}
\tilde g_p & -\tilde g_p \tilde g_q \\ 1 & -\tilde g_q 
\end{pmatrix} \frac{\lambda \alpha_{pq}}{\tilde g_q-\tilde g_p} \; , \]
and so both $F$ and 
\[ F \cdot \begin{pmatrix}
\bar a & - b \\ \bar b & a 
\end{pmatrix} \] will produce the same discrete 
CMC $1$ surface in $\mathbb{H}^3$, but will give different linear 
Weingarten families.  
\end{itemize}

\begin{remark}\label{discreteairyduality}
When taking $g$ (resp. $\hat g$) to be the discrete power function 
$z^\gamma$ (resp. $-z^{\hat \gamma}$) as in \eqref{eqn:alphagnm} 
and \eqref{eqn-agafonovaxes}, Equation 
\eqref{eqn:geometryandsomething} will hold with $\alpha_{pq}=1$ on 
horizontal edges and $\alpha_{pq}=-1$ on vertical edges if 
\[ \gamma+\hat \gamma = 2 \] 
(see Lemma \ref{lem-app8} in Appendix \ref{appendix}).  
Thus by the third item above, $g$ and $\hat g$ will produce two 
discrete flat surfaces that are parallel 
to each other.  We saw this same behavior in the smooth 
case as well, see Remark \ref{smoothairyduality}.  
\end{remark}

We have the following result, which can be proven 
similarly to the way Theorem \ref{thm:concirc-quads} was proven.  

\begin{theorem}
For any $t$, the resulting linear Weingarten surface $f_t$ 
has concircular quadrilaterals.  
\end{theorem}

\begin{remark}
Note again that, like in Remark \ref{rem:nonisothermicity}, 
when $t$ is not $1$, the surface will not be discrete 
isothermic in general.
\end{remark}

\subsection{Examples}

We now give three discrete examples, in parallel with the previous 
smooth examples \ref{examplewithqconstant}, 
\ref{smoothsurfacesofrev} and \ref{thirdsmoothexample}.  The 
third example is in the next section \ref{section4}.  

\begin{example}\label{example4pt11}
Here we discretize Example \ref{examplewithqconstant}.  
We define, for any nonzero 
constant $q>0$, \[ g_{m,n} = q \cdot (m+i n) \; , \] and then 
\[ (g_{m+1,n}-g_{m,n}) (g_{m+1,n+1}-g_{m+1,n})^{-1} 
(g_{m,n+1}-g_{m+1,n+1}) (g_{m,n}-g_{m,n+1})^{-1} = -1 \; , \] 
so we can define the cross ratio factorizing function as 
(where $n_p$ denotes the $n$ coordinate of $p \in D$) 
\[ \alpha_{pq} = (-1)^{n_p+n_q} \; . \]
We take a solution $F$ of Equation \eqref{discretebryeqn}.  
We have
\[ F_{m+1,n} = F_{m,n} U_{m,n} \; , \;\;\; F_{m,n+1} = 
F_{m,n} V_{m,n} \; , \] 
where
\[ U_{m,n} = 
\begin{pmatrix}
1+\lambda (m+in) & -\lambda q (m+in)(1+m+in) \\ 
\lambda q^{-1} & 1-\lambda (1+m+in)
\end{pmatrix} \; , \]\[ V_{m,n} = 
\begin{pmatrix}
1+\lambda (im-n) & -i \lambda q (m+in)(i+m+in) \\ 
i \lambda q^{-1} & 1-\lambda (im-n-1) \end{pmatrix} \; . 
\]
We have the necessary compatibility condition 
\[ V_{m,n} U_{m,n+1} = U_{m,n} V_{m+1,n} \; . \]
Then, using this $F_{m,n}$, we can construct the Weingarten 
family for this holomorphic function $g_{m,n}$.  There are 
special isolated values of the scaling $\lambda$ that give 
atypical results, just like in the smooth case in Example 
\ref{examplewithqconstant}, where $|q|=1$ gave the atypical 
result of a geodesic line for the resulting flat "surface".  
However, usually the resulting flat surface 
will be a discrete cylinder.  
\end{example}

\begin{figure}[phbt]
  \centering
  \includegraphics[scale=0.4]{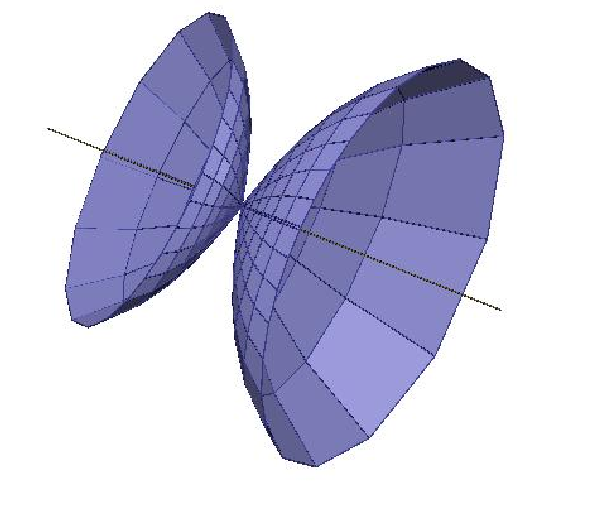}
  \includegraphics[scale=0.55]{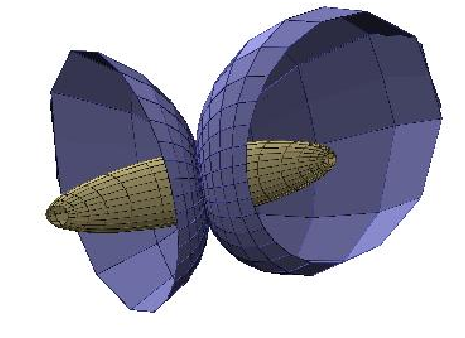}
  \caption{Discrete flat surfaces of revolution and their caustics, 
           shown in the Klein model for $\mathbb{H}^3$.  The left surface 
           is called a "snowman", and the one on the right is called 
           an "hourglass".}
  \label{uglyfigure8}
\end{figure}

\begin{example}\label{example4pt12}
Like in the smooth case (Example \ref{smoothsurfacesofrev}), we can 
use discrete holomorphic 
exponential functions $g_{m,n} = e^{c(m+in)}$ for 
any nonzero constant $c \in \mathbb{R} \cup 
(i \mathbb{R})$ to construct discrete 
flat surfaces of revolution.  See Figure \ref{uglyfigure8}.  
\end{example}

\begin{remark} In Figure \ref{uglyfigure8}, we use the Klein model 
for $\mathbb{H}^3$.  We sometimes find the Klein ball model to 
           be more convenient than the Poincare ball model, because 
           geodesics in the Klein model are the same as Euclidean 
           straight lines.  Since we are dealing with discrete 
           surfaces with geodesic edges, this can be convenient.  
           This is particularly useful when looking at the 
           intersection set of two discrete surfaces.
\end{remark}

\section{An example related to the Airy equation, and 
Stokes phenomenon}\label{section4}

\begin{figure}[phbt]
  \centering
  \includegraphics[scale=0.32]{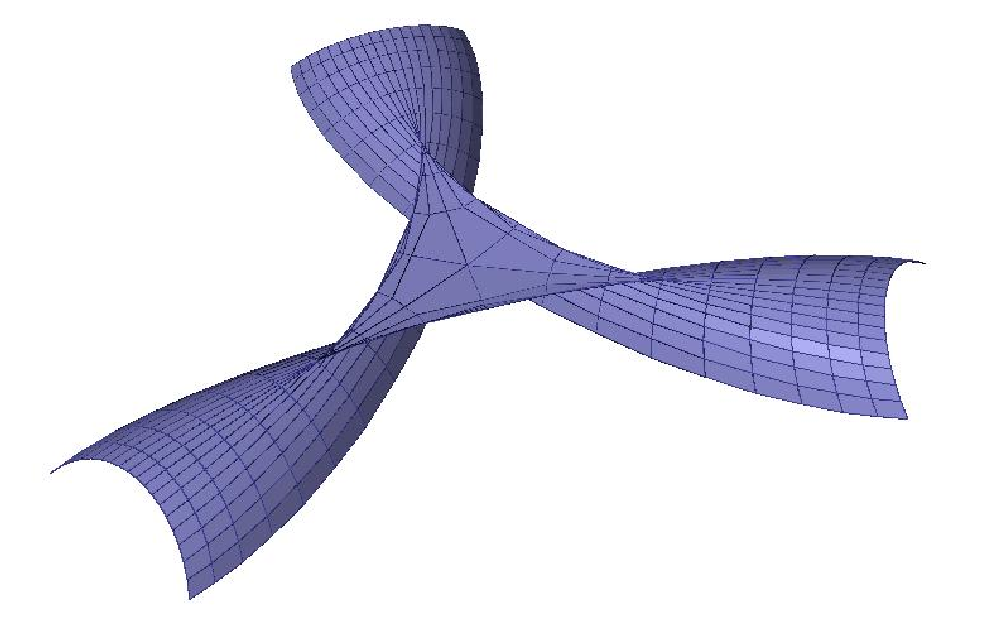}
  \includegraphics[scale=0.32]{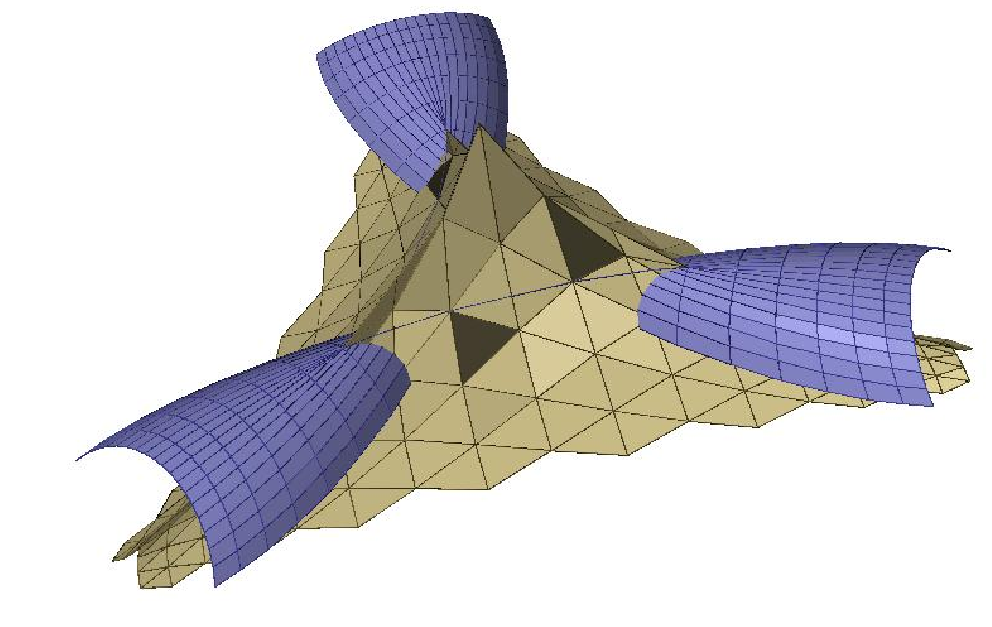}
  \caption{A discrete 
     flat surface in the Klein model constructed from the discrete 
     holomorphic function $z^{4/3}$, and thus related to the 
     Airy equation.  The surface is on the left, and the 
     surface together with its discrete caustic is on the right.  
     Because the discrete version of 
     $z^\gamma$ is constructed from circle packings, every fourth vertex 
     has two adjacent horizontal (resp. vertical) edges 
     for which $|dg=g_q-g_p|$ is the same.  Since $|dg|$ 
     determines the distance to focal points, we always have three normal 
     geodesics off vertices of the surface meeting at a single focal point. 
     It follows that $C_f$ consists of triangles in this case.}
  \label{uglyfigure6}
\end{figure}

The following Example \ref{foruglyfigure6} 
is of interest because it has similar properties 
to the corresponding surface in the smooth case: 
the surface in the smooth case has trifold 
symmetry and has three swallowtail singularities connected by three 
cuspidal edges, and also has a Stokes phenomenon in its asymptotic 
behavior \cite{SY} (see Example \ref{thirdsmoothexample}).  

\subsection{The discrete flat surface made with the discrete 
power function $z^{4/3}$}
We now give an example related to the Airy equation.  

\begin{example}\label{foruglyfigure6}
For this discrete example (see Figure \ref{uglyfigure6}), we need 
the discrete version of $g=z^{4/3}$, as defined in Section 
\ref{sect:2pt2}.  
Because we have a discrete holomorphic function $z^{4/3}$, 
we are in a position to be able to consider the 
resulting discrete flat surface via Equations 
\eqref{discretebryeqn}, \eqref{EpEpEp} and 
\eqref{starstar1}, and also a discrete Airy equation, analogous 
to Equation \eqref{eqn:q}.  

Numerical evidence (see Figure 
\ref{uglyfigure6}) suggests similar corresponding 
swallowtail singularities for this discrete flat surface, 
analogous to those 
for the smooth case in Example \ref{thirdsmoothexample}.  

In Figure \ref{stokes}, portions of the hyperbolic Schwarz 
image (i.e. this discrete flat surface) related to the Airy equation 
are shown.  The image 
of the discrete half-line $z>0$ ($z \in \mathbb{Z}$) is drawn in red,
which has a single limit point in the boundary sphere.
The image of the discrete half-line $iz<0$ ($z \in i \mathbb{Z}$) is drawn in blue, 
which has no single limit point in the boundary sphere and 
instead wraps around infinitely many times.  This behavior 
is typical for the continuous Airy case along the Stokes direction, 
and the corresponding curves on the smooth flat surface associated 
with the Airy equation behave in the same way 
(see Example \ref{thirdsmoothexample}).  This provides a numerical 
confirmation of a Stokes phenomenon for the discrete Airy function.  

We also note that the numerics suggest a behavior 
similar to that of the smooth case regarding singularities, as the 
image looks from a distance as though it has three cuspidal edge arcs 
connecting at three swallowtail singularities.  Furthermore the 
discrete caustic of this flat surface 
has a similar behavior to that of the corresponding 
smooth caustic, in that it points sharply outward at its center 
of symmetry.  (Caustics will be introduced in the next section.)  
\end{example}

\begin{figure}[h]
\begin{center}
 \includegraphics[width=5cm]{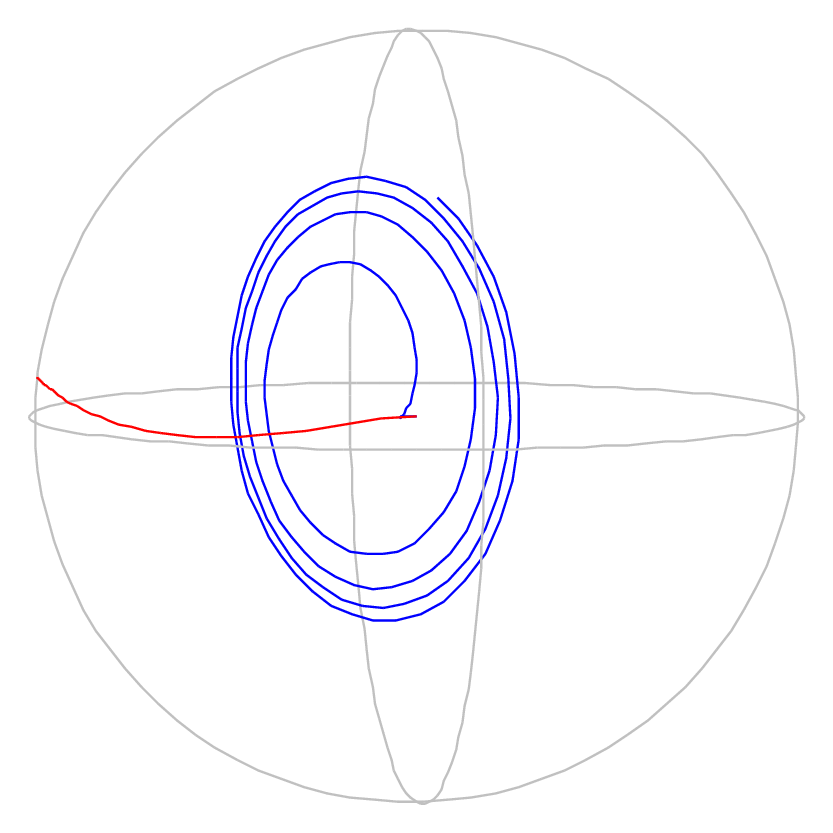}
\end{center}
\caption{Asymptotic limit of two curves in the discrete 
flat surface associated with the Airy equation.}
\label{stokes}
\end{figure}

\begin{figure}[h]
\begin{center}
\begin{tabular}{cc}
$t=0$\quad (flat) & $t=1/10$  \\
  \includegraphics[width=5cm]{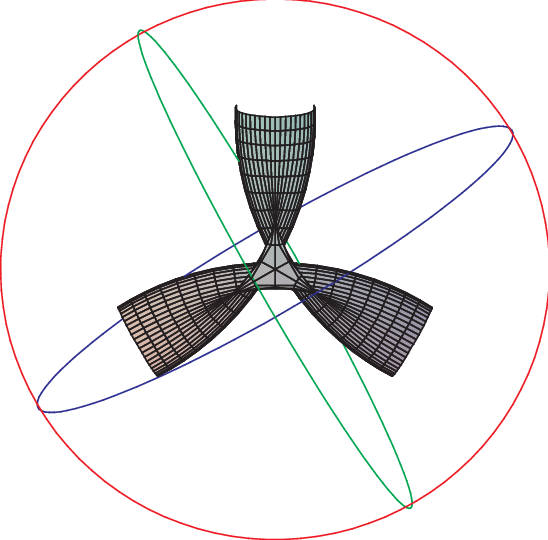} \qquad &
\qquad  \includegraphics[width=5cm]{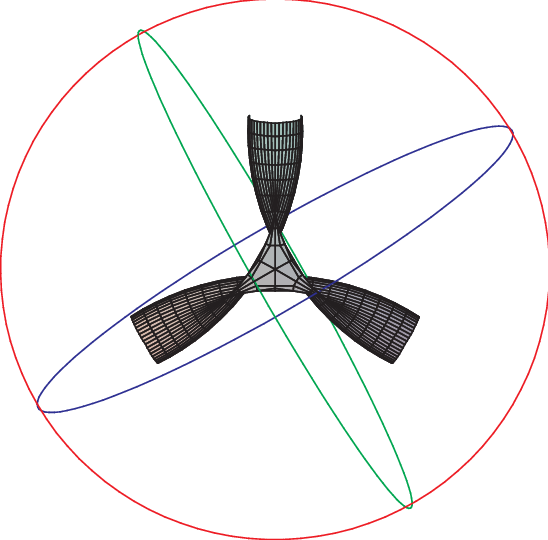} \\
$t=1/2$ & $t=1$\quad (CMC1) \\
  \includegraphics[width=5cm]{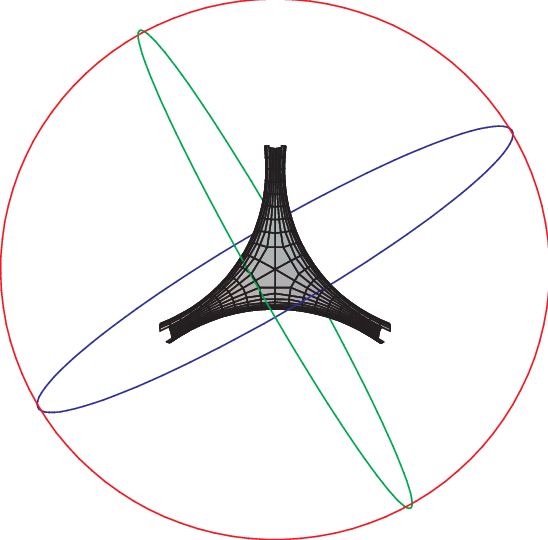} \qquad &
\qquad   \includegraphics[width=5cm]{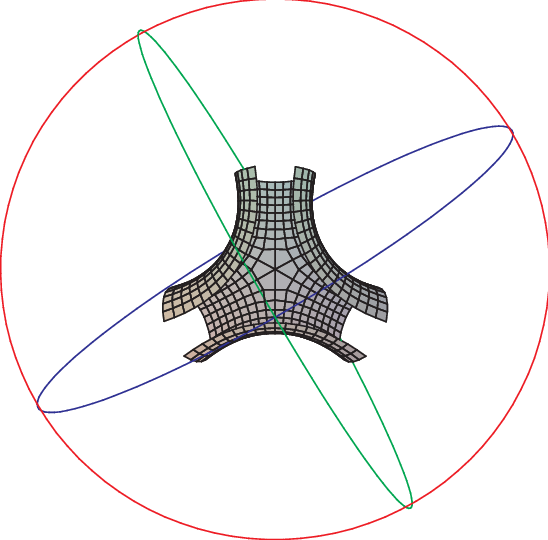} 
\end{tabular}
\end{center}
\caption{The linear Weingarten family of discrete surfaces 
for a discrete flat surface associated with the Airy equation, 
in the Poincare ball model.  (The upper-left surface here 
is the same as the left-hand surface in Figure 
\ref{uglyfigure6}.)}
\label{discairycmc}
\end{figure}

\section{Caustics of discrete flat surfaces}
\label{section5}

In this and the next section (and in Appendix \ref{appendix9}), 
since we will consider flat surfaces and their caustics 
exclusively, we will abbreviate the notation "$f_0$" to "$f$".  

\subsection{Definition of discrete caustics}
Let $C_{f}$ 
be the collection of focal points, i.e. the {\em focal surface}, 
also called 
the {\em caustic}, of a smooth flat surface $f$ in 
$\mathbb{H}^3$ (note that we 
should assume $f$ is a flat front, so that the caustic will exist).  
If $E_{f}$ is the lift of $f = E_{f} \bar E_{f}^T$ 
(determined from $g$), then (see \cite{KRSUY}) 
\begin{equation}\label{star29} E_{C_{f}} = E_{f} 
\cdot \begin{pmatrix}
\sqrt{g^\prime} & 0 \\ 0 & \frac{1}{\sqrt{g^\prime}} 
\end{pmatrix} \cdot P \; , \;\;\;\;\; P = \frac{1}{\sqrt{2}}\begin{pmatrix}
1 & \sqrt{-1} \\ \sqrt{-1} & 1 
\end{pmatrix} \; , \end{equation}
is a lift of $C_{f} = E_{C_{f}} \bar E_{C_{f}}^T$.  
Although we have just described the caustic in terms of 
Weierstrass data, it is independent of the choice of that data.  
We know that $C_{f}$ is a flat surface (see \cite{KRSUY}, \cite{KRUY1}), 
because its lift $E_{C_{f}}$ satisfies the following equation: 
\begin{equation}\label{caustic-pot-smooth} 
(E_{C_{f}})^{-1} dE_{C_{f}} = 
\begin{pmatrix}
0 & 1 + \frac{\sqrt{-1}}{2} \frac{g^{\prime\prime}}{g^\prime} \\ 
1 - \frac{\sqrt{-1}}{2} \frac{g^{\prime\prime}}{g^\prime} & 0 
\end{pmatrix} dz \; . \end{equation}

For the case of a discrete flat surface $f$ 
in $\mathbb{H}^3$ with discrete lift $E$, 
we must first consider how to define the caustic $C_{f}$.  We can 
define the normal $N_p$ as in \eqref{flatsurfnormal2} 
at each vertex $f_{p}$ of $f$, so we have normal 
geodesics emanating from each vertex, and we can consider 
when normal geodesics of adjacent vertices will intersect.  Once we 
have those intersection points, we will see that we can consider 
them as vertices of $C_{f}$, giving us a definition for 
$C_{f}$.  

\begin{lemma}\label{single-caustic-lemma}
Let $f$ be a discrete flat surface in $\mathbb{H}^3$ with lift $E$, 
as in \eqref{star15}, constructed 
using the discrete holomorphic function $g$.  Let $\alpha_{pq}$ be 
a cross ratio factorizing function for $g$.  Then the normal geodesics 
in $\mathbb{H}^3$ emanating from two adjacent vertices $f_{p}$ and 
$f_{q}$ will intersect 
if and only if $\lambda \alpha_{pq} < 0$, in which case the 
intersection point is unique and is equidistant from $f_{p}$ 
and $f_{q}$.  
Furthermore, even when the two normal geodesics do not intersect, 
they still lie in a single common geodesic plane.  
\end{lemma}

\begin{proof}
Consider one edge $pq$.  Applying an isometry of 
$\mathbb{H}^3$ if necessary, 
we may assume without loss of generality that 
\[ E_p = I \; , \;\; f_{p} = \frac{1}{\det E_p} E_p \bar E_p^T = 
I \; , \;\; N_p = 
\begin{pmatrix} 1 & 0 \\ 0 & -1 \end{pmatrix} \; . \]  
Then, with $dg_{pq}=g_q-g_p$, 
\[ E_q = \begin{pmatrix} 1 & dg_{pq} \\ \frac{\lambda 
\alpha_{pq}}{dg_{pq}} & 1 \end{pmatrix} \; , \;\; 
\det E_q = 1 - \lambda \alpha_{pq} \; . \]
Thus \[ f_{q} = \frac{1}{1-\lambda \alpha_{pq}} E_q \bar E_q^T \;\; 
\text{and} \;\; N_q = \frac{1}{1-\lambda \alpha_{pq}} E_q 
\begin{pmatrix} 1 & 0 \\ 0 & -1 \end{pmatrix} \bar E_q^T \; . \]  
The condition for the two normal geodesics to intersect is that there 
exist reals $t_1$ and $t_2$ so that 
\[ \cosh t_1 \cdot \begin{pmatrix} 1 & 0 \\ 0 & 1 \end{pmatrix} 
+ \sinh t_1 \cdot \begin{pmatrix} 1 & 0 \\ 0 & -1 \end{pmatrix} = 
\frac{\cosh t_2}{1-\lambda \alpha_{pq}} 
E_q \bar E_q^T + \frac{\sinh t_2}{1-\lambda \alpha_{pq}} 
E_q \begin{pmatrix} 1 & 0 \\ 0 & -1 \end{pmatrix} \bar E_q^T \; . \]
In other words (we now abbreviate $dg_{pq}$ to $dg$, and 
$\alpha_{pq}$ to $\alpha$), \[ (1-\lambda \alpha) 
\begin{pmatrix} \cosh t_1 + \sinh t_1 & 0 \\ 0 & 
\cosh t_1 - \sinh t_1 \end{pmatrix} = \cosh t_2 \cdot 
E_q \bar E_q^T + \sinh t_2 \cdot 
E_q \begin{pmatrix} 1 & 0 \\ 0 & -1 \end{pmatrix} \bar E_q^T \]
\[ = \cosh t_2 \cdot \begin{pmatrix} 1+|dg|^2 & 
\frac{\lambda \alpha}{\overline{dg}} + dg \\ 
\frac{\lambda \alpha}{dg} + \overline{dg} & 1 + 
\frac{\lambda^2 \alpha^2}{|dg|^2} 
\end{pmatrix}
+ \sinh t_2 \cdot \begin{pmatrix} 1-|dg|^2 & 
\frac{\lambda \alpha}{\overline{dg}} - dg \\ 
\frac{\lambda \alpha}{dg} - \overline{dg} & 
\frac{\lambda^2 \alpha^2}{|dg|^2} - 1 
\end{pmatrix} \; . \] 
There exists a $t_2$ so that this last sum on the 
right-hand side is a diagonal matrix if and only if 
\[ \left| \frac{\lambda \alpha - |dg|^2}{\lambda \alpha + 
|dg|^2} \right| > 1 \; . \] 
So the normal lines intersect in $\mathbb{H}^3$ if and only if 
\[ | \lambda \alpha - |dg|^2 | > | \lambda \alpha + |dg|^2 | \; . \]  
This is equivalent to 
\[ \lambda \alpha < 0 \; , \] 
and then $t_2$ satisfies 
\[ \sinh(t_2) = \frac{|dg|^2+\lambda \alpha}{\sqrt{-4 \lambda 
\alpha} |dg|} \; , \;\;\; 
\cosh(t_2) = \frac{|dg|^2-\lambda \alpha}{\sqrt{-4 \lambda 
\alpha} |dg|} \; . \]  Now, to get the diagonal terms in the 
above matrix equation to match, we want $t_1$ such that 
\[ \cosh ( t_1 ) = \frac{1}{2(1-\lambda \alpha)} \left[ 
\left(2+|dg|^2+\frac{\lambda^2 \alpha^2}{|dg|^2}\right) \cosh t_2 + 
\left(\frac{\lambda^2 \alpha^2}{|dg|^2}-|dg|^2\right) 
\sinh t_2 \right]\; , \]  
\[ \sinh ( t_1 ) = \frac{1}{2(1-\lambda \alpha)} \left[ 
\left(2-|dg|^2-\frac{\lambda^2 \alpha^2}{|dg|^2}\right) \sinh t_2 + 
\left(|dg|^2-\frac{\lambda^2 \alpha^2}{|dg|^2}\right) 
\cosh t_2 \right]\; . \]  
Such a $t_1$ does exist, and in fact a computation 
shows that this $t_1$ is equal to $t_2$.  

Now the sign of $t_2=t_1$ determines which side of the quadrilateral 
the intersection lies on, and $t_2=t_1 > 0$ if and only if 
$|dg|^2 + \lambda \alpha > 0$.  

Note that because $t_1=t_2$, the distance from the intersection 
point of the normal lines to either of $f_{p}$ and 
$f_{q}$ is the same.  

By a further isometry of $\mathbb{H}^3$ that preserves 
$f_{p}$ and $N_p$, we may 
change $g$ to $e^{i \theta} g$ for some constant $\theta \in \mathbb{R}$.  
Thus without loss of generality we may assume $dg_{pq} \in \mathbb{R}$.  
It is then clear from the above equations that the two 
geodesics emanating in the normal directions from 
$f_{p}$ and $f_{q}$ both lie 
in the geodesic plane $\{ x_2 = 0 \} \cap 
\mathbb{H}^3$ of $\mathbb{H}^3$, with $\mathbb{H}^3$ 
represented as in Section \ref{section3pt1}.  This proves the last 
claim of the lemma.  
\end{proof}

\begin{remark}\label{single-caustic-lemma-but-in-R31}
By an argument similar to that of the proof of Lemma 
\ref{single-caustic-lemma}, but simpler, we also have the 
following statements: 
Let $f$ be a discrete flat surface in $\mathbb{H}^3$ with lift $E$, 
as in \eqref{star15}, constructed 
using the discrete holomorphic function $g$.  Let $\alpha_{pq}$ be 
a cross ratio factorizing function for $g$.  
Let $N_p$ denote the normal as in \eqref{flatsurfnormal2} 
at $f_{p}$.  Then the lines in $\mathbb{R}^{3,1}$ (not $\mathbb{H}^3$) 
emanating from two adjacent vertices $f_{p}$ and 
$f_{q}$ in the directions of $N_p$ and $N_q$ 
(respectively) will either be parallel or will 
intersect at a unique point that is equidistant from $f_{p}$ 
and $f_{q}$.  The distance from either $f_{p}$ or $f_{q}$ 
to that intersection point is 
\[ \left|
\frac{|g_q-g_p|^2+\lambda \alpha_{pq}}{|g_q-g_p|^2-\lambda \alpha_{pq}} 
\right| \; . \] 
This is true on each edge $pq$, 
regardless of the sign of $\lambda \alpha_{pq}$.  
Furthermore, these two lines stemming from $f_{p}$ and $f_{q}$ will 
not be parallel if \[ \lambda \neq \pm 
\frac{|g_q-g_p|^2}{\alpha_{pq}} \; . \]  
\end{remark}

Now, for all that follows, we introduce the following assumption. 

\begin{quote}
{\bf Assumption:} For $g$ the discrete holomorphic map used to construct a 
discrete flat surface, assume that all quadrilaterals in the image 
of $g$ in the complex plane are properly embedded.  
\end{quote}

This implies the following properties:

\begin{enumerate}
\item $dg_{pq}$ is never zero for any edge $\overline{pq}$ between 
     adjacent vertices $p$ and $q$ in the domain $D$.  Note that this 
     was already assumed in Section \ref{sect:2pt3}.  
\item $dg_{pq}$ and $dg_{ps}$ are never parallel for any 
      square of edge-length $1$ with vertices $p,q,r,s$ in $D$.  
   (If $dg_{pq}$ and $dg_{ps}$ 
   are parallel, then $g_p,g_q,g_r,g_s$ all lie in one line, which 
   implies that the interior of the quadrilateral in the complex plane 
   is a half-plane, which is not properly embedded.)
\end{enumerate}

Because the cross ratio for $g$ is always negative, the term 
$\lambda \alpha_{pq}$ is negative on exactly all of the horizontal 
edges in the domain $D \subset \mathbb{Z}^2$ of $f$, or exactly all of the 
vertical edges.  Applying a $90$ degree rotation to the domain if 
necessary, we can, and now do, assume without loss of generality that 
$\lambda \alpha_{pq}$ is negative if and only if the edge $pq$ is 
vertical in $D$.  

Thus every vertical edge $pq$ provides a unique intersection point 
of the normal geodesics, which we denote by $(C_{f})_{pq}$.  The domain 
$D_C$ of this new mesh is now the collection of all vertical edges of $D$.  
If we let the vertical edges be represented by their midpoints, then 
we can consider $C_{f}$ to also be defined on a square grid, but now a 
shifted one, i.e. $D_C \subset \mathbb{Z} \times (\mathbb{Z}+\tfrac{1}{2})$.  

\begin{definition}
We call the discrete surface 
\[ C_{f}=\{(C_{f})_{pq} \, | \, 
pq \; \text{is a vertical edge} \} \] the {\em caustic}, 
or {\em focal surface}, of $f$.  
\end{definition}

This new discrete surface $C_f$ is generally not a discrete flat surface, 
as the vertices in the quadrilaterals of the caustic will generally not 
be concircular.  
However, it has a number of interesting properties closely related to 
discrete flat surfaces, which we will describe.  The first property is 
stated in this lemma: 

\begin{lemma}\label{lem6pt6}
The quadrilaterals of $(C_{f})_{pq}$ lie in geodesic planes of 
$\mathbb{H}^3$.  
\end{lemma}

\begin{proof}
This is clear from the geometry of the construction, and from the fact that 
adjacent normal geodesics always lie in a common geodesic plane of 
$\mathbb{H}^3$, 
by Lemma \ref{single-caustic-lemma}. 
\end{proof}

\subsection{Discrete extrinsic curvature: first approach}
Now that we have seen the proof of Lemma \ref{single-caustic-lemma}, 
we are able to 
give one geometric justification for why we can say the discrete 
surfaces in Section \ref{section4pt3} are "flat".  We will give 
an argument here showing the "discrete extrinsic curvature" is 
identically $1$.  

For a smooth flat surface given by a frame $E$ solving 
\eqref{E-eqn-for-flat-guys}, we find that the two functions 
\[ k_1 = \frac{-1+|g_x|^2}{1+|g_x|^2} \; , \;\;\; 
   k_2 = \frac{1+|g_y|^2}{-1+|g_y|^2} \;  \] 
give the principal curvatures of the surface, and also give the 
inverses of the distances to the focal points in 
$\mathbb{R}^{3,1}$ of the curves in the 
surface along which either $y$ or $x$ is constant.  Here we 
are considering the focal points found in $\mathbb{R}^{3,1}$ (not 
$\mathbb{H}^3$), but we are finding those focal points 
with respect to the normal directions to the 
surface in $\mathbb{H}^3$ given by the normal vectors 
\eqref{flatsurfnormal}, which are actually tangent vectors to 
$\mathbb{H}^3$ itself.  Since the extrinsic 
curvature $k_1 k_2$ is exactly $1$, if $k_1 \neq k_2$, we have 
$|k_1| < 1$ and $|k_2| > 1$, and we have that 
\begin{equation}\label{artanhguy}
\text{arctanh}\left( k_1 \right) - 
\text{arctanh}\left( k_2^{-1} \right) = 
\text{arctanh}\left(\frac{-1+|g_x|^2}{1+|g_x|^2}\right) - 
\text{arctanh}\left(\frac{-1+|g_y|^2}{1+|g_y|^2}\right) = 0 \; . 
\end{equation} 
This last right-hand equality, of course, 
is clear from the fact that $|g_x|=|g_y|$.  However, the 
right-hand equality 
encodes that the extrinsic curvature is exactly $1$ in a 
way that can be applied to the discrete case, as follows: 
the corresponding equation in the discrete case is given by 
the corresponding summation about the four edges 
(assume $\lambda \alpha_{pq} = \lambda \alpha_{rs} > 0$ and 
$\lambda \alpha_{qr} = \lambda \alpha_{sp} < 0$ -- the other 
case can be handled similarly) 
\begin{equation}\label{artanhguydiscrete}
\text{arctanh}\left(\frac{-\lambda \alpha_{pq}+
|g_q-g_p|^2}{\lambda \alpha_{pq} 
+ |g_q-g_p|^2}\right) - \text{arctanh}\left(\frac{\lambda \alpha_{qr}+
|g_r-g_q|^2}{-\lambda \alpha_{qr} + |g_r-g_q|^2}\right) + \end{equation}\[
\text{arctanh}\left(\frac{-\lambda \alpha_{rs}+
|g_s-g_r|^2}{\lambda \alpha_{rs} 
+ |g_s-g_r|^2}\right) - \text{arctanh}\left(\frac{\lambda \alpha_{sp}+
|g_p-g_s|^2}{-\lambda \alpha_{sp} + |g_p-g_s|^2}\right) = 0 \] 
of each quadrilateral with vertices associated to 
$p,q,r,s$ (given in counterclockwise order about the quadrilateral 
in $D$) in the discrete surface.  
The analogous geometric meaning of 
\[ \frac{- \lambda \alpha_{pq}+|g_q-g_p|^2}{\lambda \alpha_{pq} 
+ |g_q-g_p|^2} \] is preserved in the discrete case, as it is 
the inverse of the (oriented) 
distance from either $f_{p}$ or $f_{q}$ to the intersection point 
of the geodesics in $\mathbb{R}^{3,1}$ stemming off of $f_{p}$ and $f_{q}$ 
in the directions of $N_p$ and $N_q$, respectively (see Remark 
\ref{single-caustic-lemma-but-in-R31}).  Furthermore, 
Equation \eqref{artanhguydiscrete} follows immediately from the 
definition of the cross ratio factorizing function $\alpha$.  
In this sense, we can say that the ``discrete extrinsic curvature'' 
is identically $1$.  

\subsection{Discrete extrinsic curvature: second approach}
We now consider a second approach to discrete extrinsic curvature.  
For a smooth surface of constant extrinsic curvature $1$, the 
infinitesimal ratio of the area of the Gauss map to the area of 
the surface is exactly $1$.  So another way to give a notion 
that the "discrete extrinsic curvature" be identically $1$ for a 
discrete flat surface is to show the analogous property in 
the discrete case.  That is the purpose of the following lemma.  

\begin{lemma}
Let $f$ be a discrete flat surface in $\mathbb{H}^3$ 
with normal map $N$.  Consider the vertices $f_{p},f_{q},f_{r},f_{s}$ 
of one quadrilateral (in $\mathbb{H}^3$) of the 
surface associated with the quadrilateral with vertices $p$, $q$, $r$, 
$s$ (given in counterclockwise order) in $D$.  
These four vertices $f_{p},f_{q},f_{r},f_{s}$ also determine 
another quadrilateral $\mathcal{F}_{f}$, now in $\mathbb{R}^{3,1}$, again 
with vertices $f_{p},f_{q},f_{r},f_{s}$, but now with geodesic 
edges $\overline{f_{p}f_{q}}$, $\overline{f_{q}f_{r}}$, 
$\overline{f_{r}f_{s}}$, $\overline{f_{s}f_{p}}$ in 
$\mathbb{R}^{3,1}$, and which is planar in $\mathbb{R}^{3,1}$.  Likewise, 
the normals $N_p,N_q,N_r,N_s$ determine a planar quadrilateral 
$\mathcal{F}_N$ in 
$\mathbb{R}^{3,1}$ with vertices 
$N_p,N_q,N_r,N_s$ and with geodesic edges $\overline{N_pN_q}$, 
$\overline{N_qN_r}$, $\overline{N_rN_s}$, $\overline{N_sN_p}$ 
in $\mathbb{R}^{3,1}$.  

These two quadrilaterals $\mathcal{F}_{f}$ and $\mathcal{F}_N$ 
lie in parallel spacelike 
planes of $\mathbb{R}^{3,1}$ and have the same area.  
\end{lemma}

\begin{proof}
Because $f_{p},f_{q},f_{r},f_{s}$ 
lie in a circle $\mathcal{C}$ in $\mathbb{H}^3$, 
there exists a planar quadrilateral $\mathcal{F}_{f}$ in $\mathbb{R}^{3,1}$ 
with edges that are geodesics in $\mathbb{R}^{3,1}$, and with vertices 
$f_{p},f_{q},f_{r},f_{s}$.  

Since $N_p$ and $N_q$ have reflective symmetry with respect 
to the edge of $\mathcal{F}_{f}$ from $f_{p}$ and 
$f_{q}$ (see Remark \ref{single-caustic-lemma-but-in-R31}), and 
since similar symmetry 
holds on the other three edges of $\mathcal{F}_{f}$, we 
know that $f_{p}+N_p,f_{q}+N_q,f_{r}+N_r,f_{s}+N_s$ are 
the vertices of a 
planar quadrilateral $\mathcal{F}_{f+N}$ with geodesic edges in 
$\mathbb{R}^{3,1}$.  
It follows that $N_p,N_q,N_r,N_s$ are then the vertices of a 
planar quadrilateral $\mathcal{F}_N$ with geodesic edges in 
$\mathbb{R}^{3,1}$.  
In fact, $\mathcal{F}_{f}$, $\mathcal{F}_{f+N}$ and $\mathcal{F}_N$ 
all lie in parallel spacelike planes.  

The goal is to show that $\mathcal{F}_{f}$ and $\mathcal{F}_N$ have the 
same area.  Since $\mathcal{F}_{f}$ and $\mathcal{F}_N$ are parallel, 
it is allowable to replace the metric of $\mathbb{R}^{3,1}$ with the 
standard positive-definite Euclidean metric for 
$\mathbb{R}^4$ and simply prove that $\mathcal{F}_{f}$ and $\mathcal{F}_N$ 
have the same area with respect to that metric.  
The advantage of this is that it allows us to use known computational 
methods involving mixed areas.  See \cite{BS}, for example, 
for an explanation of mixed areas.  

Noting that $f \pm N$ lies in the $3$-dimensional light cone of 
$\mathbb{R}^{3,1}$, 
i.e. $\langle f+N,f+N \rangle = \langle f-N,f-N \rangle = 0$, 
we define the lightlike vectors 
\[ G_1 = \tfrac{1}{2} (f+N) \; , \;\;\; G_2 = \tfrac{1}{2} (f-N) \; . \]
(In fact, the normal geodesic at each vertex of $f$ in $\mathbb{H}^3$ 
is asymptotic to the lines in the light cone determined by $G_1$ and 
$G_2$.)  
The two concircular sets $\{ (G_1)_p,(G_1)_q,(G_1)_r,(G_1)_s \}$ 
and $\{ (G_2)_p,(G_2)_q,(G_2)_r,(G_2)_s \}$ have the same real cross ratio, 
so the two quadrilaterals in $\mathbb{R}^{3,1}$ that they determine are either 
congruent or dual to each other.  In fact, they are dual to each other, 
seen by examining the four distances to intersection points amongst the 
normal lines in $\mathbb{R}^{3,1}$ extending from $f_p$, $f_q$, $f_r$ and 
$f_s$ (two of which to one side of $\mathcal{F}_{f}$ are less than 
$1$, and the other two of which to the opposite side are greater 
than $1$, see the proof of Lemma \ref{single-caustic-lemma} and 
also Remark \ref{single-caustic-lemma-but-in-R31}).  
It follows that the mixed area of $G_1$ and $G_2$ is zero \cite{BS}.  
We then have, with ``$A$'' denoting area and ``$MA$'' 
denoting mixed area, 
\[ A(f) = A(G_1+G_2) = A(G_1) + A(G_2) + 2 MA(G_1,G_2) = A(G_1) + A(G_2) \; , 
\]
\[ A(N) = A(G_1-G_2) = A(G_1) + A(G_2) - 2 MA(G_1,G_2) = A(G_1) + A(G_2) \; . 
\]
Hence $A(f)=A(N)$.  
\end{proof}

\begin{remark}
The proof of Lemma \ref{single-caustic-lemma} shows that for a vertical 
edge $\overline{pq}$ (of length $1$) 
of $D$, the geodesic edge $\overline{f_{p}f_{q}}$, 
the geodesic through $f_{p}$ in the direction of $N_p$ and the 
geodesic through $f_{q}$ in the direction of $N_q$ form the boundary 
of a planar equilateral triangle in $\mathbb{H}^3$.  In particular, it follows 
that $f_{p}$, $f_{q}$, $f_{q}^d$ and $f_{p}^d$ are 
concircular, for any value of $d$.  One can also 
show that $f_{p}$, $f_{q}$, $f_{q}^d$ and $f_{p}^d$ are 
concircular even when 
$\overline{pq}$ is a horizontal edge in $D$.  This shows that the 
quadrilaterals formed by the two points of an edge and the two points  
of the corresponding edge of a parallel flat surface are 
always concircular. This  
in turn implies that a quadrilateral of the surface and the corresponding 
quadrilateral on a parallel surface have a total of eight vertices all 
lying on a common sphere - forming a cubical object 
with concircular sides. This gives (see \cite{BMS}) a discrete 
version of a triply orthogonal system, that is, a map from $\mathbb{Z}^3$ or a 
subdomain of $\mathbb{Z}^3$ to $\mathbb{R}^3$ where all quadrilaterals are concircular.
\end{remark}

\subsection{A formula for the caustic}

In Lemma \ref{lem6pt6} we gave one property of caustics that is closely 
related to discrete flat surfaces.  Here we give a second such type of 
property, as seen in Theorem \ref{anotherthm-forcaustic} below.  

The equation for the lift $E$ of $f$ is 
\[ E_p^{-1} E_q = \begin{pmatrix}
1 & dg_{pq} \\ \frac{\lambda \alpha_{pq}}{dg_{pq}} & 1
\end{pmatrix} \; . \]  
However, since the formula \eqref{starstar1} for the surface has a 
mitigating scalar factor $1/\det E$, we can change the 
equation above so that the 
potential matrix has determinant one, without changing the resulting 
surface, so let us instead use: 
\[ \tilde E_p = \frac{1}{\sqrt{\det E_p}} E_p \; , \;\;\; 
f_{p} = \tilde E_p \overline{\tilde E_p}^T \; , \]
\[ \tilde E_p^{-1} 
\tilde E_q = \frac{1}{\sqrt{1-\lambda \alpha_{pq}}} \begin{pmatrix}
1 & dg_{pq} \\ \frac{\lambda \alpha_{pq}}{dg_{pq}} & 1
\end{pmatrix} \; . \]  
We also now assume that $\lambda$ is sufficiently close to zero so that 
\[ | \lambda \alpha_{pq} | < 1 \]  
for all edges $pq$.  

We now define, for each vertical edge $pq$, 
\begin{equation}\label{6pt1discretized} 
E_{(C_{f})_{pq}} = (a \tilde E_p+b \tilde E_q) \cdot 
\begin{pmatrix}
\frac{\sqrt{dg_{pq}}}{\sqrt[4]{\lambda \alpha_{pq}}} & 0 
\\ 0 & \frac{\sqrt[4]{\lambda \alpha_{pq}}}{\sqrt{dg_{pq}}} 
\end{pmatrix} \cdot P \; , \end{equation} 
where $a$ and $b$ are any choice of 
nonnegative reals such that $a+b=1$ (recall the definition of 
$P$ in \eqref{star29}).  This is a natural 
discretization of the $E_{C_{f}}$ for the case of smooth surfaces 
(see \eqref{star29}), where we 
now must take a weighted average of $E_p$ and $E_q$, and we allow 
any choice of weighting $(a,b)$.  
For the smooth case in Equation \eqref{E-eqn-for-flat-guys}, 
the fourth root of the upper right 
term of $E^{-1}dE$ divided by the lower left term gives the 
$\sqrt{g^\prime}$ appearing in Equation \eqref{star29}.  
For the discrete case in Equation 
\eqref{star15}, the fourth root of the upper right 
term of $E_p^{-1}(E_q-E_p)$ divided by the lower left term gives the 
$\sqrt{dg_{pq}}/\sqrt[4]{\lambda \alpha_{pq}}$ appearing in 
Equation \eqref{6pt1discretized} here.  This explains why we insert the 
$\sqrt[4]{\lambda \alpha_{pq}}$ factors here.  Note that $\sqrt[4]{\lambda 
\alpha_{pq}}$ is not real, because $\lambda \alpha_{pq} < 0$.  

\begin{theorem}\label{anotherthm-forcaustic}
The formula 
\begin{equation}\label{star38} 
C_{f} = \frac{1}{\det (E_{C_{f}})} E_{C_{f}} \cdot 
\overline{E_{C_{f}}}^T 
\end{equation} 
for the discrete caustic holds for all vertical edges $pq$, 
and this formula does not depend on the choice of $a$ and $b=1-a$.  
\end{theorem}

\begin{proof}
A computation gives 
\[ (E_{C_{f}})_{pq} \overline{(E_{C_{f}})_{pq}}^T = S \cdot \tilde E_p 
\begin{pmatrix} \frac{|dg_{pq}|}{\sqrt{-\lambda \alpha_{pq}}} & 0 \\ 0 & 
\frac{\sqrt{-\lambda \alpha_{pq}}}{|dg_{pq}|} \end{pmatrix} 
\bar{\tilde{E}}_p^T = \frac{S}{\det E_p} \cdot 
E_p \begin{pmatrix} \frac{|dg_{pq}|}{\sqrt{-\lambda \alpha_{pq}}} & 0 \\ 0 & 
\frac{\sqrt{-\lambda \alpha_{pq}}}{|dg_{pq}|} \end{pmatrix} \bar E_p^T 
\; , \]\[ S = 1+2 a b \frac{1-\sqrt{1-\lambda 
\alpha_{pq}}}{\sqrt{1-\lambda \alpha_{pq}}} \; . \]  
The scalar factor $S$ is the only part of 
$(E_{C_{f}})_{pq} \overline{(E_{C_{f}})_{pq}}^T$ that 
depends on $a$ and $b$, but this scalar factor is irrelevant in the 
formula \eqref{star38}, so we see independence from the choice of $a$ 
and $b$.  The result now follows from the proof of 
Lemma \ref{single-caustic-lemma}.  
\end{proof}

\begin{remark}
However, the choice of normal direction at the vertices of $C_f$ does 
depend on the choice of $a$ and $b$, as the following equation shows: 
\begin{equation}\label{rats1} 
N_{pq} = (\tilde E_{C_{f}})_{pq} \begin{pmatrix} 1 & 0 \\ 0 & -1 
\end{pmatrix} \overline{(\tilde E_{C_{f}})_{pq}}^T = \tilde E_p 
\Omega \bar{\tilde{E}}_p^T \; , \end{equation} where $\Omega$ is 
\[ \begin{pmatrix} 
\frac{2b}{\sqrt{1-\lambda \alpha_{pq}}} (a+ 
\frac{b}{\sqrt{1-\lambda \alpha_{pq}}}) |dg_{pq}| & 
\left( \left(a+\frac{b}{\sqrt{1-\lambda \alpha_{pq}}} \right)^2 + 
  \frac{b^2}{1-\lambda \alpha_{pq}} \lambda \alpha_{pq} \right)
\frac{\sqrt{dg_{pq}}}{\overline{\sqrt{dg_{pq}}}}
\\ 
\left( \left(a+\frac{b}{\sqrt{1-\lambda \alpha_{pq}}} \right)^2 + 
  \frac{b^2}{1-\lambda \alpha_{pq}} \lambda \alpha_{pq} \right)
\frac{\overline{\sqrt{dg_{pq}}}}{\sqrt{dg_{pq}}}
& 
\frac{2b}{\sqrt{1-\lambda \alpha_{pq}}} (a+ 
\frac{b}{\sqrt{1-\lambda \alpha_{pq}}}) \frac{\lambda 
\alpha_{pq}}{|dg_{pq}|} 
\end{pmatrix} \; . \] 
Equation \eqref{rats1} implies that, 
for both horizontal and vertical edges, the 
adjacent normal geodesics of the caustic typically do not 
intersect, for any generic choice of $a$ and $b$.  
\end{remark}

\section{Singularities of discrete flat surfaces}
\label{section6}

The purpose of the next results is to show that the 
discrete caustics in Section \ref{section5} have properties 
similar to the caustics in the smooth case.  

In what follows, we will regard both the discrete flat surface $f$ 
and its caustic $C_f$ as discrete surfaces that have edges and faces in 
$\mathbb{H}^3$ (not just vertices).  
Since there is a unique geodesic line segment between 
any two points in $\mathbb{H}^3$, the edge between any two adjacent 
vertices of $f$ is uniquely determined.  The same is true of $C_f$.  
Then, since the image of the four vertices of any given fundamental 
quadrilateral in $D$ (resp. in $D_C$) under $f$ (resp. 
$C_f$) has image lying in a single geodesic plane (see Theorem 
\ref{thm:concirc-quads} and Lemma \ref{lem6pt6}), 
the image of the fundamental quadrilateral in $D$ (resp. $D_C$) 
can be regarded as a quadrilateral in a geodesic plane of 
$\mathbb{H}^3$ 
bounded by four edges of $f$ (resp. $C_f$).  This is the setting 
for the results given in this section.  

In the case of a smooth flat surface (front) $f$ in 
$\mathbb{H}^3$ and its smooth 
flat caustic $C_{f}$, every point in $C_{f}$ is a point in the singular 
set of one of the parallel flat surfaces of $f$.  
In Lemma \ref{finalprop1}, we are stating that every point in the 
discrete caustic $C_{f}$ of a discrete flat surface $f$ 
is a point in the edge set of some parallel flat surface of $f$.  
This, in conjuction with Theorem \ref{finalprop2}, suggests 
a natural candidate for the definition of the singular set of a 
discrete flat surface.  The proof of Lemma \ref{finalprop1} is 
immediate from the definitions of parallel surfaces and caustics.  

\begin{lemma}\label{finalprop1}
Let $f$ be a discrete 
flat surface defined on a domain $D \subseteq \mathbb{Z}^2$ 
determined by a discrete holomorphic 
function $g : D \to \mathbb{C}$ with properly 
embedded quadrilaterals, and let $C_{f}$ be its caustic.    
Let $P \in \mathbb{H}^3$ be 
any point in $C_{f}$, so $P$ lies in the quadrilateral 
$\mathcal{F}$ of $C_{f}$ that is determined by two 
adjacent vertices $f_{p}$, $f_{q}$ of $f$ and the normal geodesics 
(which contain two opposite edges of $\mathcal{F}$) in 
the directions $N_p$, $N_q$ at $f_{p}$, $f_{q}$, respectively.  
(Thus $\overline{pq}$ will be a horizontal edge of $D$.)  
$P$ can lie in either the interior of $\mathcal{F}$, or an 
edge of $\mathcal{F}$, or could be a vertex of $\mathcal{F}$. 

Then $P$ lies in the edge $\overline{f_{p}^d f_{q}^d}$ of 
some parallel surface $f^d$ of $f$.  
\end{lemma}

The next proposition will be used in the proof of Theorem 
\ref{finalprop2}.  

\begin{proposition}\label{finalprop3}
Let $f$ be a discrete flat surface with normal $N$ produced from a 
discrete holomorphic function $g$ with properly embedded quadrilaterals.  
Then for all vertices $p$, $N_p$ is not tangent to any quadrilateral 
of $f$ having vertex $f_{p}$.  
\end{proposition}

\begin{proof}
Take a quadrilateral with vertices $f_{p}$, $f_{q}$, 
$f_{r}$ and $f_{s}$ 
of $f$ so that $\overline{pq}$ is a horizontal edge of $D$.  
We may make all of the 
assumptions in the proof of Lemma \ref{single-caustic-lemma}, 
including the assumption that $dg_{pq}$ is real.  
Of course, the normal $N_p = \text{diag}(1,-1)$ lies 
in the hyperplane $\{ x_2 = 0 \}$ of $\mathbb{R}^{3,1}$ (here we 
regard points of $\mathbb{R}^{3,1}$ as Hermitean matrices as in 
Section \ref{section3pt1}), and so does the point
$f_{q}$, since $dg_{pq} \in \mathbb{R}$.  Furthermore, 
$\overline{pq}$ is a horizontal 
edge of $D$, so $\lambda \alpha_{pq} > 0$, which implies that $f_{q}$ 
does not lie in the geodesic in 
$\mathbb{H}^3$ containing $f_{p}$ and tangent to 
$N_p$.  However, $g$ has properly embedded 
quadrilaterals, so both $dg_{ps}$ and $\tfrac{\lambda \alpha_{ps}}{dg_{ps}} 
+ \overline{dg_{ps}}$ will not lie in $\mathbb{R}$, and thus $f_{s}$ will 
not lie in the hyperplane $\{ x_2 = 0 \}$.  It follows that $N_p$ will 
not be parallel to the geodesic plane in $\mathbb{H}^3$ containing the 
two geodesics from $f_{p}$ to $f_{q}$ and from $f_{p}$ 
to $f_{s}$.  
\end{proof}

For a smooth flat surface (front) $f$ and its 
caustic $C_{f}$, a parallel flat surface to $f$, 
including $f$ itself, 
will meet $C_{f}$ along its singular set, and that singular set is 
generally a graph in the combinatorial sense (whose edges consist 
of immersable curves).  Furthermore, 
all vertices of that combinatorial 
graph have valence at least two.  For example, 
cuspidal edges form the edges of this graph, and swallowtails give 
vertices of this graph with valence two.  In particular, no 
cuspidal edge can simply stop at some point without continuing 
on to at least one other cuspidal edge (as this would give a 
vertex of valence one).  The following 
theorem shows that an analogous property holds in the discrete case.
Note that when we are speaking of the vertices of this combinatorial 
graph in the theorem below, these vertices are {\em not} the same 
as the vertices of the discrete flat surface, nor its discrete 
caustic, in general.  

\begin{theorem}\label{finalprop2}
Let $f$ be a discrete 
flat surface defined on a domain $D \subseteq \mathbb{Z}^2$ 
determined by a discrete holomorphic function $g : D \to \mathbb{C}$ with 
properly embedded quadrilaterals, and let $C_{f}$ be 
its caustic.  Let $f^d$ be a parallel surface, and let $S_d$ be 
the set of all $P \in \mathbb{H}^3$ as in Lemma 
\ref{finalprop1}, for that value of $d$, and for any adjacent 
endpoints $p$ and $q$ of a horizontal edge of $D$.  
Assume that no two adjacent vertices of $f^d$ are ever equal, and 
that the faces of $C_{f}$ are embedded.  

Then $S_d$ is a graph (in the combinatorial sense) with edges composed of 
geodesic segments lying in the image in 
$\mathbb{H}^3$ of the horizontal edges of $D$ 
under $f^d$, and with all vertices of $S_d$ having valence at least two.  
\end{theorem}

\begin{remark}
The snowman shown on the right-hand side of Figure \ref{uglyfigure8} 
(see Example \ref{example4pt12}) provides an example to which 
Theorem \ref{finalprop2} applies.  The Airy example in Section 
\ref{section4} also satisfies the conclusion of this theorem (see 
Figure \ref{uglyfigure6}), although it does not actually satisfy the 
condition in the theorem that the faces of the caustic be embedded.
\end{remark}

\begin{remark}
At least one of the assumptions in Theorem \ref{finalprop2} that the 
quadrilaterals of $C_{f}$ are embedded and that no two adjacent vertices 
of $f^d$ coincide is necessary.  Without them, the 
discrete hourglass, as seen in Figure \ref{uglyfigure8} 
(see Example \ref{example4pt12}), 
would provide a counterexample to the result.  
However, it is still an open question whether both of those conditions 
are really needed.  There are reasons 
why it is not obvious that we can remove one of those two conditions.  
We explore those reasons in Appendix \ref{appendix9}.  
\end{remark}

\begin{figure}[phbt]
  \centering
  \includegraphics[scale=0.75]{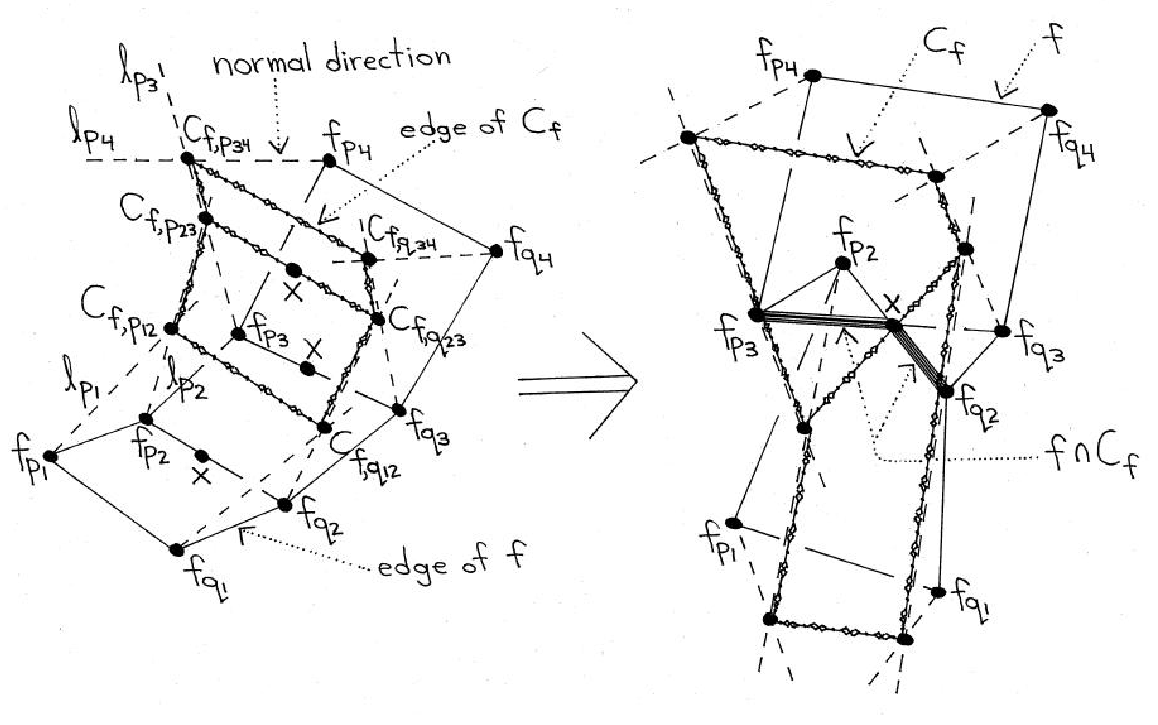}
  \vspace{-0.2in}
  \caption{A graphical representation of the argument in the proof 
            of Theorem \ref{finalprop2}.}
  \label{uglyfigure1}
\end{figure}

\begin{proof}
We must show that all vertices of $S_d$ have valence at least two.  
There are essentially only two situations for which we need to 
show this, one obvious and one not obvious.  The obvious case is 
when an entire edge $\overline{f_{p}^df_{q}^d}$ lies in $S_d$, 
and then the result is clear.  The non-obvious case that we now 
describe, where only a part of $\overline{f_{p}^df_{q}^d}$ 
lies in $S_d$, is essentially only one situtation, since 
any other non-obvious situation can be reformulated in terms 
of the notation given below.  Without loss of generality, replacing 
$f^d$ by $f$ if necessary, we may assume that $d=1$.  

Let $D$ be a domain in $\mathbb{Z}^2$ containing $p_1=(0,1)$, 
$p_2=(0,2)$, $p_3=(0,3)$, $p_4=(0,4)$, $q_1=(1,1)$, 
$q_2=(1,2)$, $q_3=(1,3)$, 
$q_4=(1,4)$, and let $g$ be a discrete holomorphic function 
defined on $D$.  Let $f=f^1$ be the resulting discrete flat surface.  
We have a normal direction defined at each vertex of $f$, which 
determines normal geodesics $\ell_{p_i}$, $\ell_{q_i}$ at the 
vertices $f_{p_i}$, $f_{q_i}$, respectively.  Note that 
$\ell_{p_i}$ and $\ell_{q_i}$ never intersect (and thus $f_{p_i}$ and 
$f_{q_i}$ are never equal), although they do lie in 
the same geodesic plane, and that 
$\ell_{p_i}$ and $\ell_{p_{i+1}}$ (resp. 
$\ell_{q_i}$ and $\ell_{q_{i+1}}$) intersect at a single point that we call 
$C_{f,p_{i,i+1}}$ (resp. $C_{f,q_{i,i+1}}$), by 
Lemma \ref{single-caustic-lemma}.  The point $C_{f,p_{i,i+1}}$ 
(resp. $C_{f,q_{i,i+1}}$) is equidistant from the two vertices 
$f_{p_i}$ and $f_{p_{i+1}}$ (resp. 
$f_{q_i}$ and $f_{q_{i+1}}$), as in Lemma \ref{single-caustic-lemma}. 

Now the caustic $C_f$ has two quadrilaterals, described here by 
listing their vertices in order about each quadrilateral: 
\[ \mathcal{F}_1=(C_{f,p_{12}},C_{f,q_{12}},C_{f,q_{23}},C_{f,p_{23}}) 
\;\;\; \text{and} \;\;\; 
\mathcal{F}_2 = (C_{f,p_{23}},C_{f,q_{23}},C_{f,q_{34}},C_{f,p_{34}}) \; . 
\]  
Let $x$ be a point in the geodesic edge $\overline{f_{p_2}f_{q_2}}$ 
so that $x$ also lies in the edge $\overline{C_{f,p_{23}}C_{f,q_{23}}}$.  
Since no two adjacent vertices of $f$ are ever equal, it follows that $x$ 
lies strictly in the interior of $\overline{f_{p_2}f_{q_2}}$.  
(See the left-hand side of Figure \ref{uglyfigure1}.)  Thus, since the 
face $\mathcal{F}_1$ of the caustic is embedded, there exists a 
half-open interval $\mathcal{I} = (y,x]$ or $\mathcal{I} = [x,y)$ 
contained entirely in the interior of 
$\overline{f_{p_2} f_{q_2}}$ so that $\mathcal{I}$ lies in 
$\mathcal{F}_1$.  

Thus $x$ can become a vertex of the graph $S_{d=1}$.  
We wish to show that the valence at $x$ is at least two.  This would 
mean that the visual representation is more like in the 
right-hand side of Figure \ref{uglyfigure1}, where the quadrilateral 
of $f$ with vertices $f_{p_2},f_{q_2},f_{q_3},f_{p_3}$ is nonembedded.  
It suffices to show that there exists a half-open interval 
$\tilde{\mathcal{I}} = (\tilde y,\tilde x]$ or $\tilde{\mathcal{I}} = 
[\tilde x,\tilde y)$ contained entirely in the interior of the geodesic edge 
$\overline{f_{p_3} f_{q_3}}$ so that: 
\begin{enumerate}
\item $\tilde{\mathcal{I}}$ lies in the face $\mathcal{F}_2$ 
      of the caustic, and 
\item $\tilde x = x$.
\end{enumerate}
Because $x$ lies in both the geodesic plane determined by 
$f_{p_2},f_{q_2},f_{q_3},f_{p_3}$ and the geodesic plane determined by 
$f_{p_3},f_{q_3},C_{f,{q_{23}}},C_{f,{p_{23}}}$, and because 
Proposition \ref{finalprop3} implies these two geodesic planes are not 
equal, $x$ must also 
lie in the line determined by $\overline{f_{p_3},f_{q_3}}$.  
Then, because $x$ lies in the embedded face $\mathcal{F}_1$, 
and because both $\overline{C_{f,p_{23}}C_{f,q_{23}}}$ and 
$\overline{f_{p_3}f_{q_3}}$ lie in the geodesic planar 
region between $\ell_{p_3}$ and $\ell_{q_3}$, $x$ must 
lie in the edge $\overline{f_{p_3},f_{q_3}}$ itself.  
(Note that we have now proven that
the quadrilateral with vertices $f_{p_2}$, $f_{q_2}$, $f_{q_3}$ 
and $f_{p_3}$ is not embedded, so in fact the visual representation 
must be more like in the right-hand side of Figure \ref{uglyfigure1}.)  

Keeping in mind that $x$ also 
lies in the edge $\overline{C_{f,p_{23}}, C_{f,q_{23}}}$, then 
since both edges $\overline{f_{p_3},f_{q_3}}$ and 
$\overline{C_{f,p_{23}}, C_{f,q_{23}}}$ lie in the plane determined 
by $\mathcal{F}_2$, and since $\mathcal{F}_2$ is embedded, we conclude 
existence of such an interval $\tilde{\mathcal{I}}$ with the required 
properties.  
\end{proof}

The above Lemma \ref{finalprop1} and Theorem \ref{finalprop2} 
suggest that $S_d$ has many of the right properties to make it a 
natural candidate for the singular set of 
any discrete surface $f^d$ in the parallel family of $f$.  

\begin{remark}
We have chosen to consider the set $S_d$ in the image $f(D)$ of 
a discrete flat surface, rather than in the domain $D$ itself (as is 
usually done for the singular set in the smooth case), 
because $S_d$ becomes a collection of connected curves in the image, 
while in the domain $D$ it would jump discontinuously between points 
in the lower and upper horizontal edges of quadrilaterals of $D$ (as 
we have seen in the above proof).  However, one could remedy this by 
inserting vertical lines between those lower and upper edge points in 
quadrilaterals of $D$, and then consider the set in the domain.  
\end{remark}

\section{Appendix: the discrete power function}
\label{appendix}

In Figure \ref{discairycmc}, we have drawn some of the discrete surfaces 
in the linear Weingarten family associated 
with the discretization of the Airy equation.  
For constructing these graphics, we used the discrete holomorphic 
power function.  We explain here how to solve 
the difference equation for determining the discrete power 
function, as follows:
\begin{eqnarray*}
& \text{cr}_{m,n} = 
\frac{(g_{m,n}-g_{m+1,n})(g_{m+1,n+1}-g_{m,n+1})}
{(g_{m+1,n}-g_{m+1,n+1})(g_{m,n+1}-g_{m,n})}
=-1 \; , & \\
&
\gamma g_{m,n} = 2m\frac{(g_{m+1,n}-g_{m,n})(g_{m,n}-g_{m-1,n})}
{g_{m+1,n}-g_{m-1,n}} 
+ 2n\frac{(g_{m,n+1}-g_{m,n})(g_{m,n}-g_{m,n-1})}
{g_{m,n+1}-g_{m,n-1}} \; , & \\
\end{eqnarray*}
with the initial conditions
\[g_{0,0}=0, \quad g_{1,0}=1, \quad g_{0,1}=i^\gamma \; . \]
Once we know $g_{m,0}$ and $g_{0,n}$, the full solution is
given by solving the first equation for the cross ratio.

We fix $n$ and set
\[
g_m=g_{m,n}\quad \hbox{\rm and}\quad G_m=g_{m,n+1} \; .
\]
Then, it is seen from the cross ratio condition that
\[
G_{m+1}-g_{m+1} =
{(g_{m+1}-g_m)(G_m-g_m) - (g_{m+1}-g_m)^2
\over (g_{m+1}-g_m) + (G_m-g_m)} \; .
\]
If we set
\[a_m=G_m-g_m \quad \hbox{\rm and}\quad p_m=g_{m+1}-g_m \; , \]
then $\{a_m\}$ satisfies
\[
a_{m+1}={p_ma_m-p_m^2\over a_m+p_m} \; .
\]
We define the recurrence relations:
\begin{eqnarray*}
b_{m+1} &=& p_mb_m-p_m^2c_m, \\
c_{m+1} &=& b_m + p_mc_m 
\end{eqnarray*}
so that 
\[
a_{m}={b_m\over c_m} \; .
\]
The initial conditions are
\[c_0 =1 \quad \hbox{\rm and} \quad b_0=g_{0,n+1}-g_{0,n} \; .
\]
The relation is written in the form
\[
\left(\begin{array}{c} b_{m+1}\\c_{m+1}\end{array}\right)
=
\left(\begin{array}{cc} p_m & -p_m^2 \\ 1 & p_m \end{array}\right)
\left(\begin{array}{c} b_{m}\\c_{m}\end{array}\right)
\]
and the eigenvalues of the $2 \times 2$ matrix just above 
are $(1\pm i)p_m$.

The difference equation satisfied by $c_m$ is
\[
c_{m+2}-(p_m+p_{m+1})c_{m+1}+2p_m^2 c_m =0 \qquad (c\ge 0)
\]
where $c_0=1$ and $c_1 = g_{0,n+1}+g_{1,n}-2g_{0,n}$.  
Once we have determined $\{c_m\}$, then
\[ b_m=c_{m+1}-p_mc_m \;\;\; \text{and} \;\;\; a_m={b_m\over c_m} \]
determine
\[
G_m = g_{m,n}+a_m \; .
\]
Then, using Equation \eqref{eqn-agafonovaxes}, 
we can determine $g_{m,n}$ for all 
nonegative $m$ and $n$.  

We have the following fact, which was also stated in \cite{Bob-new}: 

\begin{lemma}\label{lem-app8}
Suppose that $g_{m,n}$ is the discrete holomorphic function solving 
\eqref{eqn:alphagnm} and 
\eqref{eqn-agafonovaxes} for one choice of $\gamma$, and suppose 
$\hat g_{m,n}$ is the same, but with $\gamma$ replaced by $\hat \gamma = 
2-\gamma$.  Then $g_{m,n}$ and $-\hat g_{m,n}$ satisfy Equation 
\eqref{eqn:geometryandsomething} 
with $\alpha_{pq}=1$ (resp. $\alpha_{pq}=-1$) on horizontal (resp. 
vertical) edges.  
\end{lemma}

\begin{proof}
That \eqref{eqn:geometryandsomething} holds on the edges $(m,0)(m+1,0)$
and $(0,n)(0,n+1)$ can be easily confirmed from Equation 
\eqref{eqn-agafonovaxes}.  Then an induction argument proves the 
result on all other edges as well.  
\end{proof}

Recently, \cite{AHKM} solved this system explicitly in terms of
hypergeometric functions.  

\section{Appendix: On a maximum principle for 
discrete holomorphic functions}
\label{appendix9}

If $g(z)$ is a smooth nonconstant holomorphic 
function with respect to the usual 
complex coordinate $z$ for $\mathbb{C}$, then $\log |g|$ is a harmonic 
function, and the maximum principle for harmonic functions 
tells us that $\log |g|$ cannot have a local 
finite minimum at an interior 
point of the domain.  Thus, if $|g|$ has a local minimum at 
an interior point $z_0$, it must be that $g(z_0)=0$.  

There are various ways to discretize the notions of holomorphicity 
and harmonicity.  
See \cite{Bob-Hoff-Spring}, \cite{BMerS}, 
\cite{Bob-Spring}, \cite{BS}, \cite{steen}, \cite{Mer}, 
\cite{Pink-Polt}, \cite{Ura}, to name 
just a few of the possible references -- however, the history of 
this topic goes back much further than just the references 
mentioned here.  
These ways do provide for discrete versions of the maximum 
principle.  The definition we have chosen here for 
discrete holomorphic functions based on cross ratios, however, 
does not satisfy a particular simple-minded discrete version of the maximum 
principle, as we can see by the first explicit example below.  
A more sophisticated consideration is needed to produce a proper 
discrete version of the maximum principle, but we do not discuss that 
here, as the simplest questions are what are relevant to Theorem 
\ref{finalprop2}.  

\begin{example}\label{exa9pt1}
Set $D = \{ (m,n) \, | \, -1 \leq m \leq 1, -1 \leq n \leq 2 \}$.  
Then, with $i = \sqrt{-1}$, set 
\[ 
g_{0,-1} = \tfrac{1}{3}-6i \; , \;\;\; 
g_{0,0} = \tfrac{1}{3} \; , \;\;\; 
g_{0,1} = \tfrac{1}{3}+i \; , \;\;\; 
g_{0,2} = \tfrac{10}{3}+10i \; , \]\[ 
g_{-1,0} = \tfrac{1}{3}-\tfrac{1}{2}\sqrt{35}+\tfrac{1}{2}i 
+3 \cos (-\tfrac{2}{5} \pi)+3 i \sin (-\tfrac{2}{5} \pi) \; , \]\[ 
g_{1,0} = \tfrac{1}{3}+\tfrac{1}{2}\sqrt{15}+\tfrac{1}{2}i 
+2 \cos (-\tfrac{2}{5} \pi)+2 i \sin (-\tfrac{2}{5} \pi) \; , 
\]
and extend $g$ to all of $D$ so that all cross ratios of 
$g$ on $D$ are $-1$.  That is, take $g$ so that all 
$\alpha_{(m,n)(m+1,n)} = -1$ and all 
$\alpha_{(m,n)(m,n+1)} = 1$.  
Then we have the following two properties: 
\begin{enumerate}
\item Amongst all vertices of $D$, $|g|$ has a strict 
minimum of $1/3$ at the interior vertex $(0,0)$.  
\item Amongst all edges of $D$ between adjacent vertices 
$p$ and $q$ (both horizontal and vertical), $|g_q-g_p|$ 
has a strict minimum of $1$ at the interior edge from 
$(0,0)$ to $(0,1)$.  
\end{enumerate}
\end{example}

\begin{figure}[phbt]
  \centering
  \includegraphics[scale=0.32]{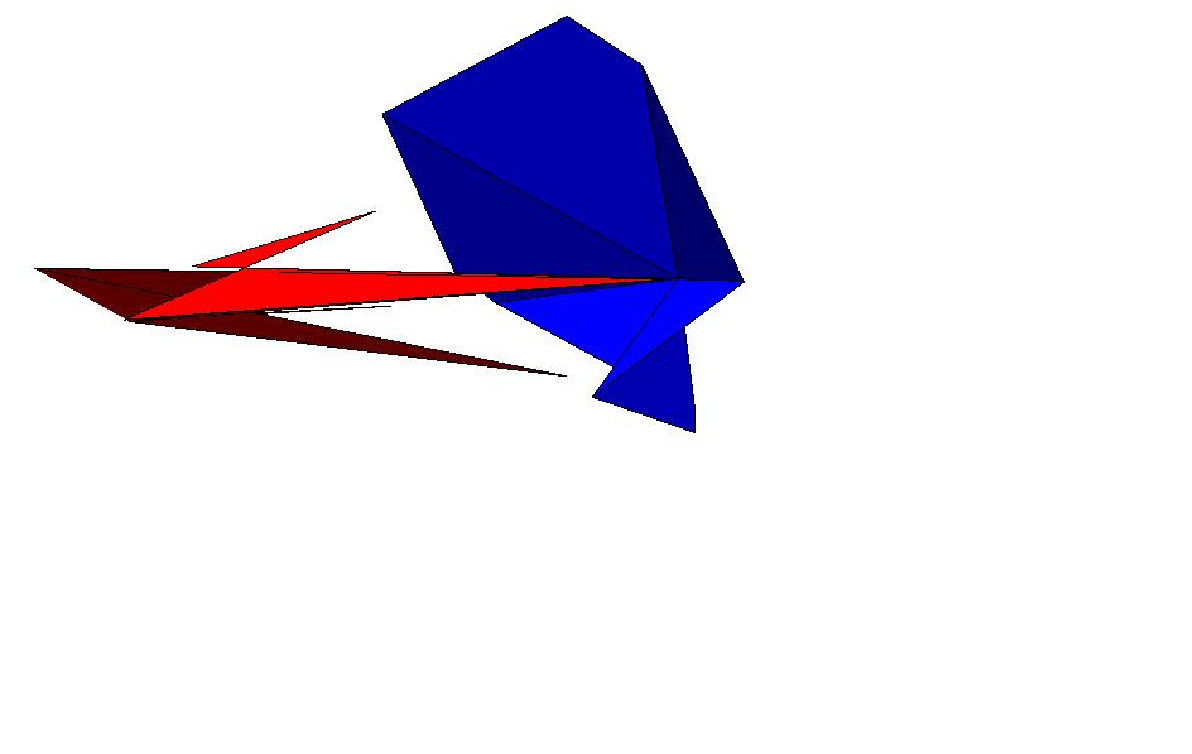}
  \includegraphics[scale=0.37]{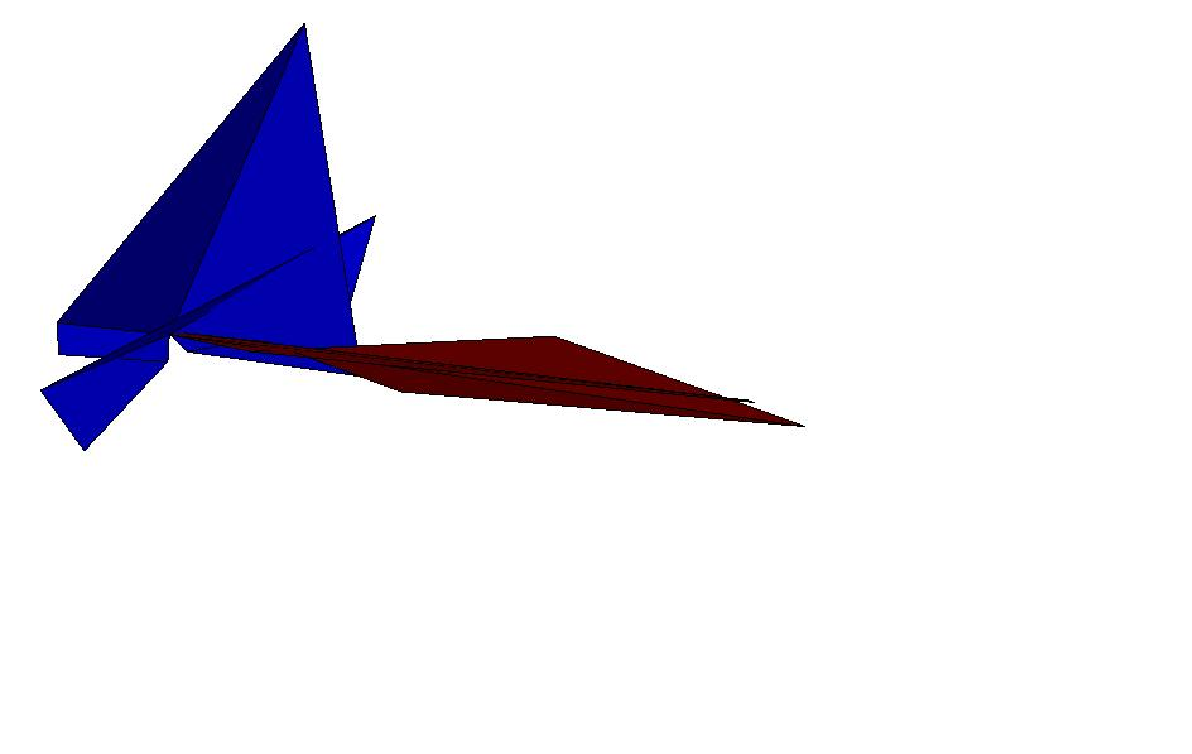}
  \vspace{-0.5in}
  \caption{Caustics $C_{f}$ which meet $f^d$ for some $d$ at 
   a single point, with $g$ taken as in Examples \ref{exa9pt1} and 
   \ref{exa9pt2}.  
   The caustics are shown in red, and the $f^d$ are shown in blue.}
  \label{fig-appendix9}
\end{figure}

Taking $g$ as in the previous example, and taking 
$\lambda =1/100$, we can produce the flat surfaces $f^d$ 
and the caustic $C_{f}$.  It turns out that the 
surface $f^d$ and the caustic $C_{f}$ intersect 
at just one point if $d \approx 1/10$ is chosen correctly.  
See the left-hand side of Figure \ref{fig-appendix9}.  
In this case, two adjacent vertices (coming from $(0,0)$ and 
$(0,1)$) of $f^d$ do coincide.  However, 
the quadrilaterals of the caustic $C_{f}$ are 
not embedded in this case, so this example does not suffice 
to show that both assumptions at the end of the first paragraph 
of Theorem \ref{finalprop2} are truly needed.  

In light of the equations in the proof of Lemma 
\ref{single-caustic-lemma}, 
a natural next step toward understanding the role of coincidence 
of vertices of $f^d$ in Theorem \ref{finalprop2} 
(and thus toward understanding 
which assumptions the theorem really needs) 
is to consider a discrete holomorphic function with the 
properties as in the next explicit example.  

\begin{example}\label{exa9pt2}
Taking the same domain $D$ as in the previous example, we now set 
\[ 
g_{0,-1} = 1-\tfrac{3}{5}i \; , \;\;\; 
g_{0,0} = 0 \; , \;\;\; 
g_{0,1} = i \; , \;\;\; 
g_{0,2} = 1+\tfrac{8}{5}i \; , \]\[ 
g_{-1,0} = -\sqrt{90}+\tfrac{1}{2}i 
+\tfrac{19}{2} \cos (-\tfrac{1}{10} \pi)+\tfrac{19}{2} i 
\sin (-\tfrac{1}{10} \pi) \; , \]\[ 
g_{1,0} = -\sqrt{6}+\tfrac{1}{2}i 
+\tfrac{5}{2} \cos (-\tfrac{3}{5} \pi)+\tfrac{5}{2} i 
\sin (-\tfrac{3}{5} \pi) \; , 
\]
and, like in the previous example, we extend $g$ to all of $D$ so 
that all $\alpha_{(m,n)(m+1,n)} = -1$ and all 
$\alpha_{(m,n)(m,n+1)} = 1$.  
Then we have the following two properties: 
\begin{enumerate}
\item Amongst all vertical edges of $D$ between adjacent vertices 
$p$ and $q$, $|g_q-g_p|$ 
has a strict minimum of $1$ at the interior edge from 
$(0,0)$ to $(0,1)$.  
\item Amongst any three vertical edges 
$pq=(-1,n)(-1,n+1)$ and $pq=(0,n)(0,n+1)$ and $pq=(1,n)(1,n+1)$ 
at the same height in $D$, $|g_q-g_p|$ 
is (strictly) minimized at the central edge $(0,n)(0,n+1)$.  
\end{enumerate}
\end{example}

Using the function $g$ in Example \ref{exa9pt2}, one might hope that 
the resulting surface $f^d$ would show the necessity of the 
assumption in Theorem \ref{finalprop2} that the adjacent vertices 
of $f^d$ do not coincide.  However, it turns out that the 
quadrilaterals of $C_{f}$ are not embedded in this case as well.  
See the right-hand side of Figure \ref{fig-appendix9}, where 
again $\lambda=1/100$ and $d$ ($\approx 1/10$) is taken so that 
the two vertices of $f^d$ coming from $(0,0)$ and $(0,1)$ 
coincide.  

Because of these subtleties, we leave open the question of whether just 
one of the two conditions in Theorem \ref{finalprop2} that 
\begin{enumerate}
\item the adjacent vertices of $f^d$ never coincide, and 
\item the caustic has embedded faces 
\end{enumerate} 
would suffice.

\end{document}